\documentclass[ams, 11pt, utopia]{nmd/nmd}
\ifnmd
\fi

\title{Twisted Alexander polynomials of hyperbolic knots}
\author{Nathan M. Dunfield}
\address{University of Illinois, Urbana, IL 61801, USA}
\email{nathan@dunfield.info}
\urladdr{http://dunfield.info}

\author{Stefan Friedl}
\address{Mathematisches Institut\\ Universit\"at zu K\"oln\\   Germany}
\email{sfriedl@gmail.com}

\author{Nicholas Jackson}
\address{University of Warwick, Coventry, UK}
\email{nicholas.jackson@warwick.ac.uk}

\def\ss{\mathfrak{s}}

\def\gl{\mbox{GL}} \def\Q{\Bbb{Q}} \def\F{\Bbb{F}} \def\Z{\Bbb{Z}} \def\R{\Bbb{R}} \def\C{\Bbb{C}}
\def\N{\Bbb{N}}  \def\l{\lambda} \def\ll{\langle} \def\rr{\rangle}
 \def\a{\alpha} \def\g{\gamma}  \def\bp{\begin{pmatrix}}
\def\sm{\setminus} \def\ep{\end{pmatrix}} \def\bn{\begin{enumerate}} \def\Hom{\mbox{Hom}}
   \def\en{\end{enumerate}}
\def\ba{\begin{array}} \def\ea{\end{array}}  
  \def\s{\sigma} \def\a{\alpha} \def\b{\beta} 
 \def\Aut{\mbox{Aut}}  \def\sign{\mbox{sign}}
  
\def\be{\begin{equation}} \def\ee{\end{equation}} \def\tr{\mbox{tr}}
   
 \def\hom{\mbox{Hom}}  
    \def\eps{\epsilon}
 \def\dim{\operatorname{dim}}

     \def\GL{\mathrm{GL}}      
   \def\K{\Bbb{K}} 
\def\w{\omega}   
  \def\gen{\mbox{genus}} \def\rt{R[t^{\pm
1}]}  \def\fr12{\frac{1}{2}} \def\z12{\Z[\fr12]} \def\rnt{R[t^{\pm 1}]^n}

\def\sltwoc{\mathfrak{sl}(2,\C)}
\def\tpm {[t^{\pm 1}]}

\def\ol{\overline}

\def\eul{\operatorname{Eul}}
\def\spinc{\operatorname{Spin}^c}

\def\ct{\C\tpm}

\def\T{\mathcal{T}}
\def\tkt{\T_K(t)}

\def\tk{\T_K}

\let\deg\undefined
\DeclareMathOperator{\deg}{deg}

\newcommand{\OO}{{\mathcal O}}
\newcommand{\interior}{\mathrm{int}}
\renewcommand{\PSL}{{\mathrm{PSL}(2, \C)}}
\renewcommand{\SL}{{\mathrm{SL}(2, \C)}}

\newcommand{\SLF}{{\mathrm{SL}(2, \F)}}
\newcommand{\SLFK}{{\mathrm{SL}(2, \F_K)}}
\newcommand{\SLOK}{{\mathrm{SL}(2, \OO_\K)}}

\newcommand{\alphabar}{\kernoverline{2}{\alpha}{2}}
\newcommand{\Tsym}[1]{\left(t^{#1} + t^{-#1}\right)}
\newcommand{\adj}{\mathit{adj}}
\DeclareMathOperator{\genus}{genus}
\newcommand{\tadjkt}{\T_K^{\adj}(t)}
\newcommand{\tadjk}{\T_K^{\adj}}
\newcommand{\Xtilde}{{\widetilde{X}}}
\newcommand{\myupbeta}{^{}{\hspace{-0.45em}\raisebox{0.95ex}{\scriptsize $\beta$}}}

\newcommand{\smalleqsize}{\scriptstyle}
\newcommand{\messylinestart}{& \ \ \smalleqsize}
\newcommand{\mvph}{\vphantom{\theta^7}}
\newcommand{\KT}{\mathit{KT}}
\newcommand{\ctil}{\tilde{c}}
\newcommand{\ctilprime}{\ctil'}

\newcommand{\tkx}{{\T_K^{\vphantom{x}} \hspace{-0.5em}\raisebox{0.8ex}{$\scriptstyle X_0$}}}
\newcommand{\tkxt}{{\tkx(t)}}
\newcommand{\cx}{{\C[X_0]}}
\newcommand{\cR}{{\C[R_0]}}
\newcommand{\cRf}{{\C(R_0)}}
\newcommand{\cxt}{{\cx\tpm}}

\newcommand{\rhotaut}{{\rho_{\mathrm{taut}}}}
\newcommand{\SLCR}{\mathrm{SL}\big(2, \cR\big)}
\newcommand{\tktaut}{{\T_K^{\mathrm{taut}}}}
\newcommand{\POne}{ P^1(\C) }
\newcommand{\FreeGroup}[1]{\left\langle#1\right\rangle}
\newcommand{\Xhat}{{ \kernhat{4}{X}{0}_0 }}
\newcommand{\ZHS}{$\Z\mathrm{HS}$}
\newcommand{\QHS}{$\Q\mathrm{HS}$}
\newcommand{\ZtwoHS}{{$\Z_2\mathrm{HS}$}}
\newcommand{\nab}{^{\mathit{irr}}}
\def\tr{\mathrm{tr}}
\newcommand{\abar}{\kernoverline{3}{a}{0}}
\newcommand{\xbar}{\kernoverline{3}{x}{0}}
\newcommand{\Abar}{\kernoverline{4}{A}{0}}
\newcommand{\gbar}{\kernoverline{3}{g}{3}}
\newcommand{\ith}{$i^{\mathit{th}}$}
\newcommand{\mth}{$m^{\mathit{th}}$}
\newcommand{\rootsofunity}[1]{{\boldsymbol \mu}_{#1}}
\newcommand{\produnity}{\prod_{\zeta \in \rootsofunity{m}}}
\newcommand{\sumunity}{\bigoplus_{\zeta \in \rootsofunity{m}}}
\newcommand{\Zm}{{\Z_m}}
\newcommand{\Ztwo}{{\Z_2}}
\newcommand{\oltk}{{{\kernoverline{2}{\T_K}{9}}}}
\newcommand{\leftFn}{\vphantom{\F}_\beta V}
\newcommand{\leftFna}{\vphantom{\F}_{\beta^*} V}
\newcommand{\rightFn}{V_\beta}
\newcommand{\op}{^{\mathit{op}}}

\theoremstyle{plain}
\newtheorem*{maintheorem}{Theorem~\ref{thm:deft}}
\newtheorem*{chartheorem}{Theorem~\ref{thm:chartorsion}}
\newtheorem*{charcorollary}{Corollary~\ref{cor:char}}
\newtheorem*{charconjecture}{Conjecture~\ref{conj:char}}

\begin{document}

\begin{abstract}
  We study a twisted Alexander polynomial naturally associated to a
  hyperbolic knot in an integer homology 3-sphere via a lift of the
  holonomy representation to $\SL$.  It is an unambiguous symmetric
  Laurent polynomial whose coefficients lie in the trace field of the
  knot.  It contains information about genus, fibering, and
  chirality, and moreover is powerful enough to sometimes detect
  mutation.

  We calculated this invariant numerically for all $313{,}209$
  hyperbolic knots in $S^3$ with at most 15 crossings, and found that
  in all cases it gave a sharp bound on the genus of the knot and
  determined both fibering and chirality.  

  We also study how such twisted Alexander polynomials vary as one
  moves around in an irreducible component $X_0$ of the
  $\SL$-character variety of the knot group.  We show how to
  understand all of these polynomials at once in terms of a polynomial
  whose coefficients lie in the function field of $X_0$.  We use this
  to help explain some of the patterns observed for knots in $S^3$,
  and explore a potential relationship between this universal
  polynomial and the Culler-Shalen theory of surfaces associated to
  ideal points.
\end{abstract}
\maketitle

\setcounter{tocdepth}{2}
\tableofcontents

\section{Introduction}

A fundamental invariant of a knot $K$ in an integral homology
\3-sphere $Y$ is its Alexander polynomial $\Delta_K$.  While
$\Delta_K$ contains information about genus and fibering, it is
determined by the maximal metabelian quotient of the fundamental group
of the complement $M = Y \setminus K$, and so this topological
information has clear limits.  In 1990, Lin introduced the
\emph{twisted Alexander polynomial} associated to $K$ and a
representation $\a \maps \pi_1(M) \to \GL(n, \F)$, where $\F$ is a
field.  These twisted Alexander polynomials also contain information
about genus and fibering and have been studied by many authors (see
the survey \cite{FV10}).  Much of this work has focused on those $\a$
which factor through a finite quotient of $\pi_1(M)$, which is closely
related to studying the ordinary Alexander polynomial in finite covers
of $M$.  In contrast, we study here a twisted Alexander polynomial
associated to a representation coming from hyperbolic geometry.

Suppose that $K$ is \emph{hyperbolic}, i.e.~the complement $M$ has a
complete hyperbolic metric of finite volume, and consider the
associated holonomy representation $\alphabar \maps \pi_1(M) \to
\Isom^+(\H^3)$.  Since $\Isom^+(\H^3) \cong \PSL$, there are two
simple ways to get a linear representation so we can consider the
twisted Alexander polynomial: compose $\alphabar$ with the adjoint
representation to get $\pi_1(M) \to \Aut(\mathfrak{sl}_2 \C) \leq
\mathrm{SL}(3, \C)$, or alternatively lift $\alphabar$ to a representation
$\pi_1(M) \to \SL$.  The former approach is the focus of the recent
paper of Dubois and Yamaguchi \cite{DY09}; the latter method is what
we use here to define an invariant $\tkt \in \C\tpm$ called the
\emph{hyperbolic torsion polynomial}.

The hyperbolic torsion polynomial $\tk$ is surprisingly little
studied.  To our knowledge it has only previously been looked at for
2-bridge knots in \cite{Mo08, KM10, HM08, SW09c}.  Here we show that
it contains a great deal of topological information.  In fact, we show
that $\tk$ determines genus and fibering for all $313{,}209$
hyperbolic knots in $S^3$ with at most 15 crossings, and we conjecture
this is the case for all knots in $S^3$.

\subsection{Basic properties}

More broadly, we consider here knots in $\Ztwo$-homology \3-spheres.  The
ambient manifold $Y$ containing the knot $K$ will always be oriented,
not just orientable, and $\tk$ depends on that orientation. Following
Turaev, we formulate $\tk$ as a Reidemeister-Milnor torsion as this
reduces its ambiguity; in that setting, we work with the compact core
of $M$, namely the knot exterior $X \assign Y \sm \mathrm{int}(N(K))$
(see Section~\ref{section:twisted} for details). By fixing certain
conventions for lifting the holonomy representation $\alphabar \maps
\pi_1(X) \to \PSL$ to $\a \maps \pi_1(X) \to \SL$, we construct in
Section~\ref{section:hyp} a well-defined symmetric polynomial $\tk \in
\C\tpm$.  The first theorem summarizes its basic properties:
\begin{theorem}\label{thm:deft}
 Let $K$ be a hyperbolic knot in an oriented $\Ztwo$-homology \3-sphere. Then
  $\tk$ has the following properties:
  \begin{enumerate}
  \item
   $\T_K$ is an unambiguous element of \hspace{0.01em} $\ct$ which
    satisfies $\tk(t^{-1}) = \tkt$. It does not depend on an
    orientation of $K$.
  \item The coefficients of $\tk$ lie in the trace field of $K$. If
    $K$ has integral traces, the coefficients of $\tk$ are algebraic
    integers.
  \item  \label{item:nonvanish-intro}
    $\tk(\xi)$ is non-zero for any root of unity $\xi$. In particular,
    $\tk \neq 0$.
  \item
    If $K^*$ denotes the mirror image of $K$, then
    $\T_{K^*}(t)=\oltk(t)$, where the coefficients of the latter
    polynomial are the complex conjugates of those of $\T_K$.
  \item \label{item:amph-intro}
    If $K$ is amphichiral then $\T_K$ is a real polynomial.
  \item The values $\tk(1)$ and $\tk(-1)$ are mutation invariant.
\end{enumerate}
\end{theorem}
Moreover, $\tk$ both determines and is determined by a sequence of
$\C$-valued torsions of finite cyclic covers of $X$.  Specifically,
let $X_m$ be the $m$--fold cyclic cover coming from the free
abelianization of $H_1(X;\Z)$.  For the restriction $\a_m$ of $\a$ to
$\pi_1(X_m)$, we consider the corresponding $\C$-valued torsion
$\tau(X_m,\a_m)$.  A standard argument shows that $\tk$ determines all
the $\tau(X_m, \a_m)$ (see Theorem~\ref{thm:tauxm}).  More interestingly,
the converse holds: the torsions $\tau(X_m,\a_m)$ determine
$\tk$ (see Theorem~\ref{prop:tktdetermined}).  This latter result
follows from work of David Fried \cite{Fr88} (see also Hillar
\cite{Hillar2005a}) and that of Menal--Ferrer and Porti \cite{MP09}.

\begin{remark}
  Conjecturally, the torsions $\tau(X_m,\a_m)$ can be expressed as
  analytic torsions and as Ruelle zeta functions defined using the
  lengths of prime geodesics.   See Ray--Singer\cite{RS71},
  Cheeger \cite{Che77,Che79}, M\"uller \cite{Mu78} and Park
  \cite{Par09} for details and background material. We hope that this
  point of view will be helpful in the further study of $\tk$.

  The torsions $\tau(X_m,\a_m)$ are
  interesting invariants in their own right. For example, Menal--Ferrer
  and Porti \cite{MP09} showed that $\tau(X_m,\a_m)$ is non--zero for
  any $m$.  Furthermore, Porti \cite{Po09} showed that
  $\tau(X_1,\a_1) = \tau(X, \a) =\tk(1)$ is not obviously related to hyperbolic
  volume. More precisely, using a variation on
  \cite[Th\'eor\`eme~4.17]{Po97} one can show that there exists a
  sequence of knots $K_n$ whose volumes converge to a positive real
  number, but the numbers $\T_{K_n}(1)$ converge to zero. See
  \cite{Po97,MP09, MenalFerrerPorti2011} for further results.
\end{remark}

\subsection{Topological information: genus and fibering}

We define $x(K)$ to be the Thurston norm of a generator of
$H_2(X, \partial X; \Z) \cong \Z$; if $K$ is null-homologous in $Y$,
then $x(K)=2 \cdot \gen(K)-1$, where $\gen(K)$ is the least genus of all
Seifert surfaces bounding $K$.  Also, we say that $K$ is
\emph{fibered} if $X$ fibers over the circle.

A key property of the ordinary Alexander polynomial $\Delta_K$ is that
\[
x(K) \geq \deg(\Delta_K) - 1.
\]
When $K$ is fibered, this is an equality and moreover the lead
coefficient of $\Delta_K$ is $1$ (here, we normalize $\Delta_K$ so
that the lead coefficient is positive).  As with any twisted
Alexander/torsion polynomial, we get similar information out of $\tk$:
\begin{theorem}\label{thm:top}
  Let $K$ be a knot in an oriented $\Ztwo$-homology sphere. Then
  \[
  x(K) \geq \frac{1}{2} \deg(\tk).
  \]
  If $K$ is fibered, this is an equality and $\tk$ is monic,
  i.e.~has lead coefficient 1.
\end{theorem}
Theorem~\ref{thm:top} is an immediate consequence the definitions
below and of \cite[Theorem~1.1]{FK06} (for the genus bound) and of the
work of Goda, Kitano and Morifuji \cite{GKM05} (for the fibered case);
see also Cha \cite{Ch03}, Kitano and Morifuji \cite{KtM05}, Pajitnov
\cite{Paj07}, Kitayama \cite{Kiy08}, \cite{FK06} and
\cite[Theorem~6.2]{FV10}.

\subsection{Experimental results}

The calculations in \cite{Ch03,GKM05,FK06} gave evidence that when one
can freely choose the representation $\a$, the twisted torsion
polynomial is very successful at detecting both $x(K)$ and non-fibered
knots.  Moreover \cite{FV11} shows that collectively the twisted
torsion polynomials of representations coming from homomorphisms to
finite groups determine whether a knot is fibered.  However, it is not
known if all twisted torsion polynomials together always detect
$x(K)$.

Instead of considering many different representations as in the
work just discussed, we focus here on a single, albeit canonical,
representation.  Despite this, we find that $\tk$ alone is a very powerful
invariant.  In Section~\ref{sec:15cross}, we describe computations
showing that the bound on $x(K)$ is sharp for all $313{,}209$
hyperbolic knots with at most 15 crossings; in contrast the bound from
$\Delta_K$ is not sharp for 2.8\% of these knots. Moreover, among such
knots $\tk$ was monic \emph{only} when the knot was fibered,
whereas 4.0\% of these knots have monic $\Delta_K$ but aren't fibered.
(Here we computed $\tk$ numerically to a precision of 250 decimal
places, see Section~\ref{sec-comp-details} for details.)

Given all this data, we are compelled to propose the following, even
though on its face it feels quite implausible, given the general
squishy nature of Alexander-type polynomials.
\begin{conjecture}\label{conj:top}
  For a hyperbolic knot $K$ in $S^3$, the hyperbolic torsion
  polynomial $\tk$ determines $x(K)$, or equivalently its genus.
  Moreover, the knot $K$ is fibered if and only if $\tk$ is monic.
\end{conjecture}
We have not done extensive experiments for knots in manifolds other
than $S^3$, but so far we have not encountered any examples where
$\tk$ doesn't contain perfect genus and fibering information.

\subsection{Topological information: Chirality and mutation}

When $K$ is an amphichiral, $\tk$ is a real polynomial (Theorem
\ref{thm:deft}(\ref{item:amph-intro})).  This turns out to be an
excellent way to detect chirality.  Indeed, among hyperbolic knots in
$S^3$ with at most 15 crossings, the 353 knots where $\tk$ is real are
\emph{exactly} the amphichiral knots (Section~\ref{sec:chiral}).

Also, hyperbolic invariants often do not detect mutation, for example
the volume \cite{Ru87}, the invariant trace field
\cite[Corollary~5.6.2]{MR03}, and the birationality type of the
geometric component of the character variety \cite{CL96,Ti00,Ti04}.
The ordinary Alexander polynomial $\Delta_K$ is also mutation
invariant for knots in $S^3$. However, $x(K)$ can change under
mutation, and given that $x(K)$ determines the degree of $\tk$ for all
15 crossing knots, it follows that $\tk$ \emph{can} change under
mutation; we discuss many such examples in
Section~\ref{sec:ex-mutants}.  However, sometimes mutation does
preserve $\tk$, and we don't know of any examples of two knots with
the same $\tk$ which aren't mutants.

\subsection{Adjoint torsion polynomial}

As we mentioned earlier, there is another natural way to get a torsion
polynomial from the holonomy $\alphabar \maps \pi_1(M) \to \PSL$,
namely by considering the adjoint representation of $\PSL$ on its Lie
algebra.  The corresponding torsion polynomial was studied by Dubois
and Yamaguchi \cite{DY09}, partly building on work of Porti
\cite{Po97}.  We refer to this invariant here as the \emph{adjoint
  torsion polynomial} and denote it $\tk^\adj$.  We also numerically
calculated this invariant for all knots with at most 15 crossings.
Unlike what we found for $\tk$, the degree of $\tk^\adj$ was
\emph{not} determined by the genus for 8{,}252 of these knots.
Moreover, we found 12 knots where the genus bound from $\tk^\adj$ was
not sharp even after accounting for the fact that $x(K)$ is
necessarily an odd integer.  The differing behaviors of these two
polynomials seems very mysterious to us; understanding what's behind
it might shed light on Conjecture~\ref{conj:top}.  See
Sections~\ref{section:adj} and \ref{sec:adjointdata} for the details
on what we found for $\tk^\adj$.

\subsection{Character varieties}

So far, we have focused on the twisted torsion polynomial of (a lift
of) the holonomy representation of the hyperbolic structure on $M$.
However, this representation is always part of a complex curve of
representations $\pi_1(M) \to \SL$, and it is natural to study how the
torsion polynomial changes as we vary the representation.  In
Sections~\ref{section:charactervariety} and \ref{sec:charexs}, we
describe how to understand all of these torsion polynomials at once, and
use this to help explain some of the patterns observed in
Section~\ref{sec:15cross}.  For the special case of 2-bridge knots,
Kim and Morifuji \cite{Mo08,KM10} had previously studied how the
torsion polynomial varies with the representation, and our results
here extend some of their work to more general knots.

Consider the character variety $X(K) \assign \Hom\big(\pi_1(M), \,
\SL\big)/\!/\,\SL$, which is an affine algebraic variety over $\C$. We
show in Section~\ref{section:charactervariety} that each $\chi \in
X(K)$ has an associated torsion polynomial $\tk^\chi$.  These
$\tk^\chi$ vary in an understandable way, in terms of a polynomial
with coefficients in the ring of regular functions $\cx$:
\begin{theorem}\label{thm:chartorsion}
  Let $X_0$ be an irreducible component of $X(K)$ which contains the
  character of an irreducible representation.  There is a unique $\tkx
  \in \cxt$ so that for all $\chi \in X_0$ one has $ \tk^\chi(t) =
  \tkx(\chi)(t)$.  Moreover, $\tkx$ is itself the torsion
  polynomial of a certain representation $\pi_1(M) \to \SLF$, and thus has
  all the usual properties (symmetry, genus bound, etc.).
\end{theorem}
\begin{corollary}\label{cor:char}
  Let $K$ be a knot in an integral homology \3-sphere.   Then
  \begin{enumerate}
  \item The set $\setdef{ \chi \in X(K)}{\deg(\tk^\chi ) = 2 x(K)}$ is Zariski open.
  \item The set $\setdef{ \chi \in X(K)}{\tk^\chi \mbox{ is monic }}$ is Zariski closed.
\end{enumerate}
\end{corollary}
It is natural to focus on the component $X_0$ of $X(K)$ which contains
the (lift of) the holonomy representation of the hyperbolic structure,
which we call the \emph{distinguished component}.  In this case $X_0$
is an algebraic curve, and we show the following conjecture is implied by Conjecture~\ref{conj:top}.
\begin{conjecture}\label{conj:char}
  Let $K$ be a hyperbolic knot in $S^3$, and $X_0$ be the distinguished
  component of its character variety.  Then $2 x(K) = \deg( \tkx )$
  and $\tkx$ is monic if and only if $K$ is fibered.
\end{conjecture}
\noindent
At the very least, Conjecture~\ref{conj:char} is true for many
2-bridge knots as we discuss in Section~\ref{sec:2bridge}.  We also
give several explicit examples of $\tkx$ in Section~\ref{sec:charexs}
and use these to explore a possible avenue for bringing the Culler-Shalen
theory of surfaces associated to ideal points of $X(K)$ to bear on Conjecture~\ref{conj:char}.

\subsection{Other remarks and open problems}

For simplicity, we restricted ourselves here to the study of knots,
especially those in $S^3$.  However, we expect that many of the
results and conjectures are valid for more general 3-manifolds.  In
the broader settings, the appropriate question is whether the twisted
torsion polynomial detects the Thurston norm and fibered classes
(see \cite{FK06,FK08} and \cite{FV08b} for more details).

We conclude this introduction with a few questions and open problems:
\bn
\item Does $\tk$ determine the volume of the complement of $K$? Some
  calculational evidence is given in \cite{FJ11} and in
  Section~\ref{sec:ex-mutants} in this paper.

\item If two knots in $S^3$ have the same $\tk$, are they necessarily
  mutants?  See Section~\ref{sec:ex-mutants} for more on this.

\item Does the invariant $\tk$ contain information about symmetries
  of the knot besides information on chirality?

\item Does there exist a hyperbolic knot with $\tk(1)=1$?

\item If $\tk$ is a real polynomial, is $K$ necessarily amphichiral?

\item For an amphichiral knot, is the top coefficient of $\tk$ always
  positive?

\item For fibered knots, why is the \emph{second} coefficient of $\tk$
  so often real?  This coefficient is the sum of the eigenvalues of
  the monodromy acting on the twisted homology of the fiber.  See
  Section~\ref{sec:otherpat} for more.

\item Why is $|\tk(-1)|>|\tk(1)|$ for $99.99\%$ of the knots
  considered in Section~\ref{sec:otherpat}?
\en

\noindent \textbf{Acknowledgments.}  We thank J\'er\^ome
Dubois, Taehee Kim, Takahiro Kitayama, Vladimir Markovic, Jinsung
Park, Joan Porti, Dan Silver, Alexander Stoimenow, Susan Williams and
Alexandru Zaharescu for interesting conversations and helpful
suggestions.  We are particularly grateful to Joan Porti for sharing
his expertise and early drafts of \cite{MP09,
    MenalFerrerPorti2011} with us.  We also thank the referee for
  helpful comments. Dunfield was
partially supported by US NSF grant DMS-0707136.

\section{Twisted invariants of 3--manifolds}
\label{section:twisted}

In this section, we review torsions of twisted homology groups and
explain how they are used to define the twisted torsion polynomial of
a knot together with a representation of its fundamental group to
$\SL$.  We then summarize the basic properties of these torsion
polynomials, including how to calculate them.

\subsection{Torsion of based chain complexes}

Let $C_*$ be a finite chain complex over a field $\F$. Suppose that
each chain group $C_i$ is equipped with an ordered basis $c_i$ and
that each homology group $H_i(C_*)$ is also equipped with an ordered
basis $h_i$.  Then there is an associated \emph{torsion invariant}
$\tau(C_*,c_*,h_*)\in \F^\times\assign\F\sm \{0\}$ as described in the
various  excellent expositions \cite{Mi66,Tu01,Tu02,Nic03}. We will
follow the convention of Turaev, which is the multiplicative inverse
of Milnor's invariant.  If the complex $C_*$ is acyclic, then we will
write $\tau(C_*,c_*)\assign\tau(C_*,c_*,\emptyset)$.

\subsection{Twisted homology}

For the rest of this section, fix a finite CW--complex $X$ and set
$\pi \assign \pi_1(X)$.  Consider a representation $\a\colon \pi\to
\GL(V)$, where $V$ is a finite-dimensional vector space over $\F$.  We
can thus view $V$ as a left $\Z[\pi]$--module.  To define the twisted
homology groups $H_*^\a(X;V)$, consider the universal cover $\Xtilde$
of $X$.  Regarding $\pi$ as the group of deck transformations of
$\Xtilde$ turns the cellular chain complex $C_*(\Xtilde) \assign
C_*(\Xtilde; \Z)$ into a left $\Z[\pi]$--module.  We then give
$C_*(\Xtilde)$ a right $\Z[\pi]$--module structure via $c \cdot g
\assign g^{-1} \cdot c$ for $c \in C_*(\Xtilde)$ and $g \in \pi$,
which allows us to consider the tensor product
\[
C_*^\a(X;V) \assign C_*(\Xtilde)\otimes_{\Z[\pi]}V.
\]
Now $C_*^\a(X;V)$ is a finite chain complex of vector spaces, and we
define $H_*^\a(X;V)$ to be its homology.

We call two representations $\a\colon \pi\to \GL(V)$ and $\b\colon
\pi\to \GL(W)$ \emph{conjugate} if there exists an isomorphism $\Psi
\maps V\to W$ such that $\a(g)=\Psi^{-1}\circ \b(g) \circ \Psi$ for
all $g\in \pi$.  Note that such a $\Psi$ induces an isomorphism of
$H_*^\a(X;V)$ with $H_*\myupbeta(X;W)$.

\subsection{Euler structures, homology orientations and twisted torsion of CW complexes}\label{section:twitorsion}\label{section:twitau}

To define the twisted torsion, we first need to introduce certain
additional structures on which it depends.  (In our final application,
most of these will come out in the wash.)  First, fix an orientation
of each cell of $X$.  Then choose an ordering of the cells of $X$ so
we can enumerate them as $c_j$; only the relative order of cells of
the same dimension will be relevant, but it is notationally convenient
to have only one subscript.

An \emph{Euler lift} for $X$ associates to each cell $c_j$ of $X$ a
cell $\ctil_j$ of $\Xtilde$ which covers it.  If $\ctilprime_j$ is another
Euler lift, then there are unique $g_j \in \pi$ so that $\ctilprime_j = g_j
\cdot \ctil_j$.  We say these two Euler lifts are \emph{equivalent} if
\[
\prod_j g_j^{(-1)^{\scriptstyle \dim( c_j )}}
 \mtext{represents the trivial element in $H_1(X;\Z)$.}
\]
An equivalence class of Euler lifts is called an \emph{Euler
  structure} on $X$.  The set of Euler structures on $X$, denoted
$\eul(X)$, admits a canonical free transitive action by
$H_1(X;\Z)$; see \cite{Tu90,Tu01,Tu02,FKK11} for more on
these Euler structures.  Finally, a \emph{homology orientation} for $X$
is just an orientation of the $\R$-vector space $H_*(X;\R)$.

Now we can define the torsion $\tau(X,\a,e,\w)$ associated to $X$, a
representation $\a$, an Euler structure $e$, and a homology orientation
$\w$.  If $H_*^\a(X;V)\ne 0$, we define $\tau(X,\a,e,\w)\assign0$, and
so now assume $H_*^\a(X;V)=0$.  Up to sign, the torsion we seek will
be that of the twisted cellular chain complex $C_*^\a(X;V)$ with
respect to the following ordered basis.  Let $\{v_k\}$ be any basis of
$V$, and $\{\ctil_j\}$ any Euler lift representing $e$.  Order the
basis $\{ \ctil_j \otimes v_k \}$ for $C_*^\a(X;V)$ lexicographically,
i.e.~$\ctil_j \otimes v_k < \ctil_{j'} \otimes v_{k'}$ if either $j <
j'$ or both $j = j'$ and $k < k'$. We thus have a based acyclic
complex $C_*^\a(X;V)$ and we can consider
\[ \tau(C_*^\a(X;V),c_*\otimes v_*)\in \F^\times.\]
When $\dim(V)$ is even, this torsion is in fact independent of all the choices
involved, but when $\dim(V)$ is odd we need to augment it as follows to
remove a sign ambiguity.

Pick an ordered basis $h_i$ for $H_*(X;\R)$ representing our homology
orientation $\w$.  Since we have ordered the cells of $X$, we can
consider the torsion
\[ \tau\big(C_*(X;\R),c_{*},h_*\big)\in \R^\times.\] We define
$\b_i(X)=\sum_{k=0}^i\dim\big(H_k(X;\R)\big)$ and $\g_i(X)
=\sum_{k=0}^i \dim\big(C_k(X;\R)\big)$, and then set $N(X) = \sum_{i}
\b_i(X)\cdot \g_i(X)$. Following \cite{FKK11}, which generalizes the
ideas of Turaev \cite{Tu86,Tu90}, we now define
\begin{multline*}
\tau(X,\a,e,\w)\assign \\
(-1)^{N(X)\cdot \dim(V)}\cdot \tau\big(C_*^\a(X;V),c_*\otimes
v_*\big)\cdot \sign\Big(\tau\big(C_*(X;\R),c_{*},h_*\big)\Big)^{\dim(V)}.
\end{multline*}
A straightforward calculation using the basic properties of
torsion shows that this invariant does not depend on any of the
choices involved, i.e.~it is independent of the ordering and orientation
of the cells of $X$, the choice of representatives for $e$ and $\w$,
and the particular basis for $V$.  Similar elementary arguments prove
the following lemma.  Here $-\w$ denotes the opposite homology
orientation to $\w$, and note that $(\det\circ \a):\pi\to \F$ factors through $H_1(X;\Z)$.
\begin{lemma} \label{lem:torondata}
  If $\b$ is conjugate to $\a$, then given $h\in H_1(X;\Z)$ and $\eps\in \{-1,1\}$, one has
\[ \tau(X,\b,h\cdot e,\eps\cdot \w)=\eps^{\dim(V)}\cdot \big( (\det \circ
\a)(h)\big)^{-1}\cdot \tau(X,\a,e,\w).\]
\end{lemma}

\subsection{Twisted torsion of 3-manifolds}\label{section:torsion3mfd}

Let $N$ be a 3-manifold whose boundary is empty or consists of
tori. We first recall some facts about $\spinc$-structures on $N$; see
\cite[Section~XI.1]{Tu02} for details. The set $\spinc(N)$ of such
structures admits a free and transitive action by $H_1(N;\Z)$.
Moreover, there exists a map $c_1 \maps \spinc(N)\to H_1(N;\Z)$ which
has the property that $c_1(h\cdot \ss)=2h+c_1(\ss)$ for any $h\in
H_1(N;\Z)$ and $\ss\in \spinc(N)$.

Now consider a triangulation $X$ of $N$.  By \cite[Section~XI]{Tu02}
there exists a canonical bijection $\spinc(N)\to \eul(X)$ which is
equivariant with respect to the actions by $H_1(N;\Z)=H_1(X;\Z)$.
Given a representation $\a\colon \pi_1(N)\to \GL(V)$, an element
$\ss\in \spinc(N)$, and a homology orientation $\w$ for $N$, we
define
\[ \tau(N,\a,\ss,\w)\assign\tau(X,\a,e,\w)\] where $e$ is the Euler
structure on $X$ corresponding to $\ss$.  It follows from the work of
Turaev \cite{Tu86,Tu90} that $ \tau(N,\a,\ss,\w)$ is independent of
the choice of triangulation and hence well-defined.  See also
\cite{FKK11} for more details about $\tau(N,\a,\ss,\w)$.

\subsection{Twisted  torsion polynomial of a knot}\label{section:twitorsionknot}\label{section:twitauknot}

Let $K$ be a knot in a rational homology \3-sphere $Y$.  Throughout,
we write $X_K \assign Y \sm \mathrm{int}(N(K))$ for the knot exterior,
which is a compact manifold with torus boundary.  We define an
\emph{orientation} of $K$ to be a choice of oriented meridian $\mu_K$;
if $Y$ is oriented, instead of just orientable, this is equivalent to
the usual notion.  Suppose now that $K$ is oriented.  Let $\pi_K \assign
\pi_1(X_K)$ and take $\phi_K\colon \pi_K\to \Z$ to be the unique
epimorphism where $\phi(\mu_K) > 0$.  There is a
canonical homology orientation $\w_K$ for $X_K$ as follows: take a
point as a basis for $H_0(X_K;\R)$ and take $\{\mu_K\}$ as the basis
for $H_1(X_K;\R)$.  We will drop $K$ from these notations if the
knot is understood from the context.

For a representation $\a\colon \pi\to \gl(n,R)$ over a
commutative domain $R$, we define a torsion polynomial as follows.
Consider the left $\Z[\pi]$--module structure on $R^n\otimes_R
\rt \cong \rnt$ given by
\[  g\cdot (v\otimes p)\assign \big(\a(g) \cdot v\big)\otimes \left(t^{\phi(g)}p\right) \]
for $g\in \pi$ and $v\otimes p \in R^n\otimes_R \rt$.
Put differently, we get a
representation $\a\otimes \phi\colon \pi\to \gl(n,\rt)$.
We denote by $Q(t)$ the field of fractions of $\rt$.
The representation  $\a\otimes \phi$ allows us to view $\rt^n$ and $Q(t)^n$ as left $\Z[\pi]$--modules.
Given $\ss\in \spinc(X)$ we define
\[ \tau(K,\a,\ss)\assign\tau(X_K,\a\otimes \phi,\ss,\w_K)\in Q(t)\]
to be the \emph{twisted torsion polynomial of the oriented knot $K$}
corresponding to the representation $\a$ and the $\spinc$-structure
$\ss$.  Calling $\tau(K,\a,\ss)$ a polynomial even though it is
defined as a rational function is reasonable given
Theorem~\ref{thm:ktm05} below.

\begin{remark}
  The study of twisted polynomial invariants of knots was introduced
  by Lin \cite{Li01}. The invariant $ \tau(K,\a,\ss)$ can be viewed as
  a refined version of the twisted Alexander polynomial of a knot and
  of Wada's invariant.  We refer to \cite{Wad94,Ki96,FV10} for more on
  twisted invariants of knots and 3-manifolds.
\end{remark}

\subsection{The  \texorpdfstring{$\SLF$}{SL(2,F)} torsion
  polynomial of a knot}\label{section:sl2torsion}

Our focus in this paper is on \2-dimensional representations, and
we now give a variant of $\tau(K,\a,\ss)$ which does not depend
on $\ss$ or the orientation of $K$.  Specifically, for an
\emph{unoriented} knot $K$ in a \QHS\ and a representation $\a \maps
\pi \to \SLF$ we define
\begin{equation}\label{eq:tordef}
 \tk^\a \assign t^{\phi(c_1(\ss))}\cdot \tau(K,\a,\ss)
\end{equation}
for any $\ss \in \spinc(X)$ and choice of oriented meridian $\mu$ and show:
\begin{theorem}\label{thm:Tsym}
  For $\a \maps \pi \to \SLF$, the invariant $\tk^\a$ is a
  well-defined element of \hspace{0.01em} $\F(t)$ which is symmetric,
  i.e.~$\tk^\a(t^{-1}) = \tk^\a(t)$.
\end{theorem}
\noindent
We will call $\tk^\a \in \F(t)$ the \emph{twisted torsion polynomial
associated to $K$} and $\a$.
\begin{proof}
  That the right-hand side of (\ref{eq:tordef}) is
  independent of the choice of $\ss$ follows easily from
  Lemma~\ref{lem:torondata} using the observation that $\det\left( (\a
      \otimes \phi ) (g) \right) = t^{2 \phi(g)}$ for $g \in \pi$
  and the properties of $c_1$ given in
  Section~\ref{section:torsion3mfd}.

  The choice of meridian $\mu$ affects the right-hand side of
  (\ref{eq:tordef}) in two ways: it is used to define the homology
  orientation $\omega$ and the homomorphism $\phi \maps \pi \to \Z$.
  The first doesn't matter by Lemma~\ref{lem:torondata}, but switching
  $\phi$ to $-\phi$ is equivalent to replacing $t$ with $t^{-1}$.
  Hence being independent of the choice of meridian is equivalent to the
  final claim that $\tk^\a$ is symmetric.

  Now any $\SLF$-representation preserves the bilinear form on
  $\F^2$ given by $(v,w)\mapsto \det(v\,w)$.  Using this observation
  it is shown in \cite[Theorem~7.3]{FKK11}, generalizing
  \cite[Corollary~3.4]{HSW10} and building on work of Turaev
  \cite{Tu86,Tu90}, that in our context we have
  \[
  \tau(K, \alpha, \ss)\left(t^{-1}\right) = t^{2 \phi(c_1(\ss) )} \cdot \tau(K,\a,\ss)
  \]
  which establishes the symmetry $\tk^\a$ and hence the theorem.
\end{proof}
While in general $\tk^\a$ is a rational function, it is frequently a
Laurent polynomial or can be computed in terms of the ordinary Alexander
polynomial $\Delta_K$.
\begin{theorem}\label{thm:ktm05}
Let $K$ be a knot in a \QHS\ and let $\a:\pi_K\to \SLF$ be a representation.
\bn
\item \label{item:irred}
  If $\a$ is irreducible, then $\tk^\a$ lies in $\F\tpm$.
\item \label{item:redone}
  If $\a$ is reducible, then $\tk^\a = \tk\myupbeta$ where
  $\beta$ is the diagonal part of $\a$, i.e.~a diagonal representation
  where $\tr\big(\b(g)\big) = \tr\big(\a(g)\big)$ for all $g \in \pi$.
\item \label{item:redtwo}
  If $\a$ is reducible and factors through $H_1(X_K; \Z)/(\mathrm{torsion})$ then
  \[ \T_K^\a(t)=\frac{\Delta_K(zt)\cdot
    \Delta_K(z^{-1}t)}{t-(z+z^{-1})+t^{-1}} \]
  where  $z,z^{-1}$ are the eigenvalues of $\a(\mu_K)$ and $\Delta_K$ is
  the symmetrized Alexander polynomial.
\en
\end{theorem}
When the ambient manifold $Y$ is an \ZHS, then the torsion polynomial
of any reducible representation $\alpha$ to $\SLF$ can be computed by
combining (\ref{item:redone}) and (\ref{item:redtwo}); when $H_1(Y;
\Z)$ is finite but nontrivial, then $\tk^\a$ is the product of the
torsion polynomials of two 1-dimensional representations, but may not
be directly related to $\Delta_K$.
\begin{proof}
  Part (\ref{item:irred}) is due to Kitano and Morifuji \cite{KtM05}
  and is seen as follows. Since $\a$ is irreducible, there is a $g \in
  [\pi, \pi]$ so that $\a(g)$ does not have trace 2 (see e.g.~Lemma
  1.5.1 of \cite{CS83} or the first part of the proof of Theorem 1.1
  of \cite{KtM05}).  Then take a presentation of $\pi$ where $g$ is a
  generator and apply Proposition~\ref{prop:fox} below with $x_i = g$;
  since $\phi(g) = 0$ and $\tr\big(\a(g)\big) \neq 2$ the denominator
  in (\ref{equ:comptau}) lies in $\F^\times$ and hence $\tk^\a$ is in
  $\F \tpm$.

  For Part (\ref{item:redone}), conjugate $\a$ so that it is upper-diagonal
  \[
  \a(g) = \mysmallmatrix{a(g)}{b(g)}{0}{a(g)^{-1}} \mtext{for all $g
    \in \pi$.}
  \]
  The diagonal part of $\a$ is the representation $\b$ given by $g \mapsto
  \mysmallmatrix{a(g)}{0}{0}{a(g)^{-1}}$.  It is easy to see,
  for instance by using (\ref{equ:comptau}), that $\tk^\a = \tk\myupbeta$.  Finally, part
  (\ref{item:redtwo}) follows from a straightforward calculation with
  (\ref{equ:comptau}), see e.g. \cite{Tu01,Tu02}.
\end{proof}

\subsection{Calculation of torsion polynomials using  Fox calculus}  \label{section:foxcalc}

Suppose we are given a knot $K$ in a \QHS\ and a representation
$\a\colon \pi_1(X_K) \to \SLF$.  In this section, we give a
simple method for computing $\tk^\a$.  As usual, we write $\pi \assign
\pi_1(X_K)$ and $\phi = \phi_K$.  We can extend the group homomorphism
$\a\otimes \phi\colon \pi\to \gl(2,\F\tpm)$ to a ring homomorphism
$\Z[\pi]\to M(2,\F\tpm)$ which we also denote by $\a\otimes \phi$.  Given
a $k\times l$--matrix $A=(a_{ij})$ over $\Z[\pi]$, we denote by
$(\a\otimes \phi)(A)$ the $2k\times 2l$--matrix obtained from $A$ by
replacing each entry $a_{ij}$ by the $2\times 2$--matrix $(\a\otimes
\phi)(a_{ij})$.

Now let $F=\ll x_1,\dots,x_{n}\rr$ be the free group on $n$ generators. By Fox
(see \cite{Fo53,Fo54,CF63} and also \cite[Section~6]{Ha05})
 there exists for each $x_i$ a \emph{Fox derivative}
\[ \frac{\partial}{\partial x_i}: F\to \Z[F] \]
with the following two properties:
\[
 \frac{\partial x_j}{\partial x_i} =\delta_{ij}  \mtext{and}
\frac{\partial (uv)}{\partial x_i} =\frac{\partial u}{\partial
  x_i}+u\frac{\partial v}{\partial x_i} \ \mbox{ for all $u,v\in F$.}
\]
We also need the involution of $\Z[F]$ which sends $g \in F$ to
$g^{-1}$ and respects addition (this is not an algebra automorphism
since it induces an anti\hyp homomorphism for multiplication).  We
denote the image of $a \in \Z[F]$ under this map by $\abar$, and if
$A$ is a matrix over $\Z[F]$ then $\Abar$ denotes the result of
applying this map to each entry.

The following allows for the efficient calculation of $\tk^\a$, since
$\pi$ always has such a presentation (e.g.~if $K$ is a knot in $S^3$
one can use a Wirtinger presentation).
 \begin{proposition}\label{prop:fox}
Let $K$ be a knot in a \QHS, and $\ll x_1,\dots,x_{n} \, |\, r_1,\dots,r_{n-1}\rr$
be a presentation of $\pi_K$ of deficiency one.  Let $A$ be the $n\times
(n-1)$--matrix with entries $a_{ij}=\frac{\partial r_j}{\partial x_i}$.
Fix a generator $x_i$ and consider the matrix $A_i$
obtained from $A$ by deleting the \ith\ row.  Then there exists an
$l\in \Z$ so that for every even-dimensional representation $\a \maps \pi \to \GL(V)$ one has
\begin{equation}
 \label{equ:comptau}
\T_K^\a(t)=t^l\cdot \frac{\det\left((\a\otimes
    \phi)\big(\Abar_i\big)\right)}{\det\left((\a\otimes \phi)\big(\xbar_i-1\big)\right)}
\end{equation}
whenever the denominator is nonzero.
\end{proposition}
\noindent
The same formula also holds, up to a sign, when $\dim(V)$ is odd.  An
easy way to ensure nonzero denominator in (\ref{equ:comptau}) is to
choose an $x_i$ where $\phi(x_i) \neq 0$; then $\det\left((\a\otimes
  \phi)\big(\xbar_i - 1\big)\right)$ is essentially the characteristic
polynomial of $\alpha(x_i)^{-1}$ and hence nonzero.

\begin{remark}
Wada's invariant (see \cite{Wad94}) is defined to be 
\[ \frac{\det\left((\a\otimes
    \phi)\big(A_i^t\big)\right)}{\det\left((\a\otimes
    \phi)\big(x_i-1\big)\right)}.\] In \cite[p.~53]{FV10} it is
erroneously claimed that, up to multiplication by a power of $t$, the
torsion polynomial $\T_K^\a$ agrees with Wada's invariant. Since there
seems to be some confusion in the literature regarding the precise
relationship between twisted torsion and Wada's invariant, we
discuss it in detail in Section~\ref{sec:conn-wada}.  In that
section, we will also see that for representations into $\SLF$, Wada's
invariant does in fact agree with $\T_K^\a(t)$.  In particular the
invariant studied by Kim and Morifuji \cite{KM10} agrees with
$\T_K^\a(t)$.
\end{remark}

Proposition~\ref{prop:fox} is an immediate consequence of:
\begin{proposition}\label{prop:fox2} Let $K, \pi, A$ be
  as above.  For each generator $x_i$, there is an $\ss \in
  \spinc(X_K)$ so that for every even-dimensional representation
  $\beta \maps \pi \to \GL(V)$ one has
  \begin{equation}\label{equ:comptau2}
  \tau(X_K, \beta, \ss) =
  \frac{\det\big(\b(\Abar_i)\big)}{\det\left(\b(\xbar_i - 1)\right)}
  \mtext{whenever the denominator is nonzero.  }
  \end{equation}
\end{proposition}
\noindent
The homology orientation $\omega$ is suppressed in
(\ref{equ:comptau2}) because by Lemma~\ref{lem:torondata} it doesn't
affect $\tau$ as $\dim(V)$ is even.
\begin{proof}[Proof of Proposition~\ref{prop:fox2}]
  Let $X$ be the canonical 2--complex corresponding to the
  presentation of $\pi$, i.e.~$X$ has one cell of dimension zero, $n$
  cells of dimension one and $n-1$ cells of dimension two.  As
  the Whitehead group of $\pi$ is trivial \cite{Wal78}, it
  follows that $X$ is simple-homotopy equivalent to any other
  CW-decomposition of $X_K$; in particular, it is simple-homotopy
  equivalent to a triangulation. By standard results (see
  e.g.~\cite[Section~8]{Tu01}) we can now use $X$ to calculate the
  torsion of $X_K$.

  Consider the Euler structure $e$ for $X$ which is given by picking
  an arbitrary lift of the vertex of $X$ to the universal cover
  $\Xtilde$, and then taking the lift of each $x_i$ which starts at
  this basepoint.  Reading out the words $r_j$ in $x_1,\dots,x_n$
  starting at the basepoint gives a canonical lift for the 2-cells
  corresponding to the relators.  With respect to this basing, the
  chain complex $C_*(\Xtilde)$ is isomorphic to the following chain
  complex
  \[
  0\to \Z[\pi]^{n-1}
  \xrightarrow{\partial_2}\Z[\pi]^n\xrightarrow{\partial_1}\Z[\pi]\to 0.
  \]
  The bases of $C_2(\Xtilde)$ and $C_1(\Xtilde)$ are abusively denoted
  by $\{r_j\}$ and $\{x_i\}$, and the basis of $C_0(\Xtilde)$ is the lifted
  basepoint $b$.  Thus
  \[
  \partial_2( r_j ) = \sum_i \frac{\partial
    r_j}{\partial x_i} x_i = \sum_i a_{ij} x_i \mtext{and} \partial_1( x_i )
  = (x_i - 1) b.
  \]

  Now fix a basis $\{v_k\}$ for $V$.  If we then view elements $v \in
  V$ as vertical vectors and $\beta(g)$ as a matrix, the left
  $\Z[\pi]$-module structure on $V$ is given by $g \cdot v = \beta(g)
  v$.  Thus in the complex $C_*\myupbeta(X;V) = C_*(\Xtilde)
  \otimes_{\Z[\pi]} V$ we have
  \begin{align*}
  \partial_2( r_j \otimes v_k) &= \sum_i \left( a_{ij} x_i \otimes v_k
  \right) = \sum_i \left( x_i \cdot \abar_{ij} \otimes v_k
  \right) = \sum_i \left( x_i \otimes \abar_{ij} \cdot v_k \right) \\
  &= \sum_i \left( x_i \otimes  \beta(\abar_{ij}) v_k \right).
  \end{align*}
  Thus with the basis ordering conventions of
  Section~\ref{section:twitau}, the twisted chain complex $C_*\myupbeta(X;V)$ is given by
    \begin{equation}\label{eq:twistcomplex}
  0\to V^{n-1}
  \xrightarrow{ \beta(\Abar)} V^n\xrightarrow{\left( \beta(\xbar_1 -
    1), \, \dots \, , \beta(\xbar_n -1)\right)} V\to 0,
  \end{equation}
  where as usual matrices act on the left of vertical vectors.

From now on, we assume that $\det\left(\b(\xbar_i - 1)\right) \neq 0$
as otherwise there is nothing to prove.  First, consider the case when
$\det\big(\beta(\Abar_i)\big) = 0$.  We claim in this case that
$C_*\myupbeta(X;V)$ is not acyclic, and thus (\ref{equ:comptau2}) holds by
the definition of $\tau$.  Consider any $v \in V^{n-1}$ which is in
the kernel of $\Abar_i$; because $\b(\xbar_i - 1)$ is non-singular, the
fact that $\partial^2 = 0$ forces $v$ to be in the kernel of $\Abar$.
Hence $H_2\myupbeta(X;V) \neq 0$ as needed.

When instead $\det\big(\beta(\Abar_i)\big) \neq 0$, then both
boundary maps in (\ref{eq:twistcomplex}) have full rank and hence the
complex is acyclic.   Following Section 2.2 of \cite{Tu01} we can use
a suitable matrix $\tau$-chain to compute the desired torsion.
Specifically \cite[Theorem~2.2]{Tu01} gives
\begin{equation}\label{eq:nearlydone}
\tau(X, \beta, e) = \frac{\det\big(\b(\Abar_i)\big)}{\det\left(\b(\xbar_i - 1)\right)}.
\end{equation}
Here, we are using that $\dim(V)$ is even, which forces the sign
discussed in \cite[Remark~2.4]{Tu01} to be positive.  Also, the
convention of \cite{Tu01} is to record a basis as the rows of a matrix,
whereas we use the columns; this is irrelevant since the determinant
is transpose invariant.   Given (\ref{eq:nearlydone}) if we take
$\ss$ be the $\spinc$--structure corresponding to $e$, we have
established (\ref{equ:comptau2}).
\end{proof}

\subsection{Connection to Wada's invariant}
\label{sec:conn-wada}

We now explain why the formula (\ref{equ:comptau}) differs from the
one used to define Wada's invariant \cite{Wad94}, and how Wada's
invariant also arises as a torsion of a suitable chain complex. To
start, suppose we have a representation $\beta \maps \pi \to \GL(d,
\F)$, where as usual $\pi$ is the fundamental group of a knot
exterior.  The representation $\beta$ makes $V \assign \F^d$ into both
a left and a right $\Z[\pi]$--module.  The left module $\leftFn$ is
defined by $g \cdot v \assign \beta(g) v$ where $v \in V$ is viewed as
a column vector, and the right module $\rightFn$ is defined by $v \cdot
g \assign v \beta(g)$ where now $v \in V$ is viewed as a row vector.

Given a left $\Z[\pi]$--module $W$ we denote by $W\op$ the right
$\Z[\pi]$--module given by $w\cdot f \assign \ol{f}\cdot w$. Similarly we can
define a left module $W\op$ for a given right $\Z[\pi]$--module $W$.
In Section~\ref{section:twisted}, we started with the \emph{left}
modules $C_*(\Xtilde)$ and $\leftFn$ and used the chain complex
\[
  C_*\myupbeta (\Xtilde, \F^d) \assign C_*(\Xtilde)\op \otimes_{\Z[\pi]} \, \leftFn
\]
when defining the torsion.

One could instead consider the chain complex
\[
\rightFn  \otimes_{\Z[\pi]} C_*(\Xtilde).
\]
Here, if $\{v_k\}$ is a basis for $V$ and $\{ \ctil_j
\}$ is a $\Z[\pi]$--basis for $C_*(\Xtilde)$, then we endow $\rightFn \otimes_{\Z[\pi]}
C_*(\Xtilde)$ with the basis $\{ v_k \otimes \ctil_j \}$ ordered \emph{reverse}
lexicographically, i.e.~$v_k \otimes \ctil_j < v_{k'} \otimes
\ctil_{j'}$ if either $j < j'$ or both $j = j'$ and $k < k'$.

Suppose now we want to compute the torsion of $\rightFn \otimes
C_*(\Xtilde)$ using the setup of the proof of Theorem~\ref{prop:fox2}.
Then we have
  \begin{align*}
  \partial_2(v_k \otimes r_j) &= \sum_i \left( v_k \otimes a_{ij} x_i \right)
  = \sum_i \left( v_k \otimes a_{ij} \cdot x_i \right)
   = \sum_i \left( v_k \cdot a_{ij} \otimes x_i \right) \\
  &= \sum_i \left( v_k \, \beta(\a_{ij}) \otimes  x_i \right).
\end{align*}
Since we are focusing on a right module $\rightFn$, it is natural to
write the matrices for the boundary maps in  $\rightFn  \otimes_{\Z[\pi]} C_*(\Xtilde)$
as matrices which \emph{act on the right of row vectors}.   With these
conventions one gets the chain complex
\[
0\to V^{n-1}
  \xrightarrow{ \beta(A^t)} V^n\xrightarrow{\left( \beta(x_1 -
    1), \, \dots \, , \beta(x_n -1)\right)^t} V\to 0
\]
where here $A^t$ denotes the transpose of $A$, and so $A^t$ is an $(n-1)
\times n$ matrix over $\Z[\pi]$.   As in the proof of
Theorem~\ref{prop:fox2}, in the generic case \cite[Theorem~2.2]{Tu01}
gives that
\[
\tau\left(\rightFn  \otimes_{\Z[\pi]} C_*(\Xtilde)\right) = \frac{\det\big(\b(A_i^t)\big)}{\det\left(\b(x_i - 1)\right)}.
\]
Up to the sign of the denominator, this is precisely the formula for
Wada's invariant given in \cite{Wad94}.

It's important to note here that $\beta(A^t)$ is not necessarily the
same as $\big(\beta(A) \big)^t$, and hence Wada's invariant may differ
from our $\tau(X, \beta)$.  However, note that there exists a
canonical isomorphism
\[ \ba{rcl} \rightFn \otimes_{\Z[\pi]} C_*(\Xtilde) &\to &
C_*(\Xtilde)\op \otimes_{\Z[\pi]} \big(\rightFn\big)\op\\
v\otimes \s&\mapsto & \s\otimes v\ea \] which moreover respects the
ordered bases.  Thus these chain complexes have the same
torsion invariant. It's easy to see that the left module
$(\rightFn)\op$ is isomorphic to $\leftFna$ where $\beta^* \maps \pi \to
\GL(d, \F)$ is the representation given by $\b^*(g) \assign
\big(\b(g)^{-1}\big)^t$.  Thus Wada's invariant for $\beta$ is our torsion
$\tau(X, \beta^*)$.

Our focus in this paper is on $\beta$ of the form $\alpha \otimes
\phi$ where $\alpha \maps \pi \to \SLF$ and $\phi \maps \pi \to \Z$ is
the usual epimorphism.  Note that $\alpha^*$ is conjugate to
$\alpha$ (see e.g. \cite{HSW10}) and hence $(\alpha \otimes \phi)^*$ is conjugate to $\alpha^*
\otimes (-\phi)$.  Since we argued in Section~\ref{section:sl2torsion}
that $\T^\a$ is independent of the choice of $\phi$, it follows that
in this case our $\T^\a$ is exactly Wada's invariant for $\alpha$.

\section{Twisted torsion of cyclic covers}
\label{sec:cyclic}

As usual, let $K$ be a knot in a \QHS\ with exterior $X$ and
fundamental group $\pi$.  For an irreducible representation $\alpha
\maps \pi \to \SL$, in this section we relate the torsion polynomial
$\tk^\a$ to a sequence of $\C$-valued torsions of finite
cyclic covers of $X$.  We show that the latter determines the former,
and will use this connection in Section~\ref{section:hyp} to prove
nonvanishing of the hyperbolic torsion polynomial.

To start, pick an orientation of $K$ to fix the homomorphism $\phi
\maps \pi \to \Z$.  For each $m \in \N$, we denote by $X_m$ the
$m$-fold cyclic cover corresponding to $\pi_m \assign \phi^{-1}(m\Z)$.
We denote by $\a_m$ the restriction of $\a$ to $\pi_m
= \pi_1(X_m)$.   Since the dimension is even and the image of $\a_m$
lies in $\SL$, it follows from Lemma \ref{lem:torondata}
that the torsion $\tau(X_m,\a,\ss,\w) \in \C$ does not depend on the choice
of $\spinc$-structure or homology orientation; therefore we denote it by $\tau(X_m,\a_m)$.  We also let $\rootsofunity{m}$ be the
set of all \mth\ roots of unity in $\C$.
The first result of this section is the following (see \cite[Corollary~27]{DY09} for a related result).
\begin{theorem}\label{thm:tauxm}
  Let $K$ be a knot in a \QHS\ with exterior $X$ and fundamental group
  $\pi$.  Let $\a\colon \pi \to \mbox{SL}(2,\C)$ be an irreducible
  representation.  Then for every $m \in \N$ we have
  \[
  \produnity \T_K^\a(\zeta)=\tau(X_m,\a_m).
  \]
\end{theorem}
\noindent
Note here since $\a$ is irreducible the torsion polynomial $\tk^\a$ is
in $\ct$ by Theorem~\ref{thm:ktm05}(\ref{item:irred}), and so
$\tk^\a(\xi)$ is well-defined for any $\xi \in \C^\times$.  
Our proof of Theorem~\ref{thm:tauxm} is inspired by an argument of
Turaev \cite[Section~1.9]{Tu86}.  Combining
Theorem~\ref{thm:tauxm} with a (generalization of) a result of David
Fried \cite{Fr88}, we will show:
\begin{theorem}\label{thm:taux}
  If $\tau(X_m,\a_m)$ is non-zero for every $m \in \N$, then the
  $\tau(X_m,\a_m)$ determine $\T_K^\a(t)\in \C(t)$.
\end{theorem}
To state the key lemmas, we first need some notation.  We denote by
$\g_m$ the representation $\pi\to \GL\big(\C[\Zm]\big)$ which is the
composite of the epimorphism $\pi \stackrel{\phi}{\to} \Z \to \Zm$
with the regular representation of $\Zm$ on $\C[\Zm]$.  Given any $\xi
\in \C^\times$, we denote by $\lambda_\xi$ the representation $\pi \to
\GL(1,\C)$ which sends $g \in \pi$ to $\xi^{\phi(g)}$.  We first prove
Theorem~\ref{thm:tauxm} assuming the following lemmas.

\begin{lemma}\label{lem:tobase}
  $\tau(X_m, \a_m) = \tau( X, \alpha \otimes \g_m)$.
\end{lemma}

\begin{lemma}\label{lem:toreval}
For every $\xi \in \C^\times$ and $\ss \in \spinc(X)$ we have
\begin{equation}\label{equ:toreval}
\tau(X,\a\otimes \phi, \ss)(\xi)= \tau(X,\a\otimes \l_{\xi},\ss).
\end{equation}
\end{lemma}

\begin{proof}[Proof of Theorem~\ref{thm:tauxm}]
Using Lemma~\ref{lem:tobase} and the fact that
$\gamma_m$ and $\sumunity\lambda_{\zeta}$ are
conjugate representations of $\pi$, we have
\[
  \tau(X_m, \a_m) = \tau( X, \alpha \otimes \g_m) =\produnity \tau(X,\a\otimes \l_{\zeta},\ss).
\]
Note here that while the other terms do not depend on $\ss$, those in
the product at right do since $\a\otimes \l_{\zeta}$ is no longer a
special linear representation.  We now apply Lemma~\ref{lem:toreval} to find
\begin{align*}
  \tau(X_m, \a_m)
  &=\produnity \tau(X,\a\otimes \l_{\zeta},\ss) =
  \produnity \tau(X,\a\otimes \phi,\ss)(\zeta) \\
  &=\produnity (\zeta)^{-\phi(c_1(\ss))}\tk^\a(\zeta) =\produnity \tk^\a(\zeta)
\end{align*}
where the last two equalities follow from (\ref{eq:tordef}) and the fact
that $\prod \zeta = 1$.
\end{proof}
\begin{proof}[Proof of Lemma~\ref{lem:tobase}]
  The idea is that for suitable choices one gets an isomorphism
  \[
  C_*^{\a_m}(X_m;V) \to  C_*^{\gamma_m \otimes \a}\big(X;\C[\Zm] \otimes_\C V\big)
  \]
  as \emph{based} chain complexes over $\C$, and hence
  their torsions are the same.

  Fix a triangulation for $X$ with an ordering $c_j$ of its cells, as
  well as an Euler lift $c_j \mapsto \ctil_j$ of the cells to
  the universal cover $\Xtilde$.  Let $\phi_m \maps \pi \to
  \Zm$ be the epimorphism whose kernel is $\pi_m = \pi_1(X_m)$, and
  fix $g \in \pi$ where $\gbar = \phi_m(g)$ generates $\Zm$.

  Consider the triangulation of $X_m$ which is pulled back from that
  of $X$, and let $c'_j$ be the cell in $X_m$ which is the image of
  $\ctil_j$ under $\Xtilde \to X_m$.  Then each cell of $X_m$ has a
  unique expression as $g^k \cdot c'_j$ for $k$ in $\{0, \cdots, k -
  1\}$, where here $g^k$ acts on $X_m$ as a deck transformation.
  We
  order these cells so that $g^k \cdot c'_j < g^{k'} \cdot c'_{j'}$ if $j <
  j'$ or both $j = j'$ and $k < k'$.  When computing torsion, we'll
  use the Euler lift $g^k \cdot c'_j \mapsto g^k \cdot \ctil_j$ for
  $X_m$.

  Let $V$ denote $\C^2$ with the $\pi$-module structure given by
  $\a$, and let  $\{v_1, v_2 \}$ be an ordered basis for $V$.
  Consider the map
  \begin{equation}\label{eq:messymap}
  f \maps C_*(\Xtilde) \otimes_{\Z[\pi_m]} V \quad  \to \quad
  C_*(\Xtilde) \otimes_{\Z[\pi]} \big( \C[\Zm] \otimes_\C V  \big)
  \end{equation}
  induced by $ \ctil \otimes v \mapsto \ctil \otimes ( 1 \otimes v )$;
  this is well defined since for $h \in \pi_m$ we have
  \begin{align*}
   f\big( (\ctil \cdot h) \otimes v \big) &= (\ctil \cdot h) \otimes (1
   \otimes v) = \ctil \otimes \big(h \cdot (1 \otimes v) \big) \\
   &= \ctil \otimes \big( (h \cdot 1) \otimes (h \cdot v) \big) =
   \ctil \otimes \big( 1 \otimes (h \cdot v) \big) = f\big( \ctil
   \otimes (h \cdot v) \big),
  \end{align*}
  where we used $h \in \pi_m$ to see $h \cdot 1 = 1$ in $\C[\Zm]$.
  Clearly $f$ is a chain map of complexes of $\C$-vector spaces, and
  it is an isomorphism since it sends the elements of the basis
  $\{(g^k \cdot \ctil_j) \otimes v_\ell\}$ to those of the basis
  $\{\ctil_k \otimes ( \gbar^{-k} \otimes g^{-k} \cdot v_\ell)\}$.
  Now choose $\{ v_{k,\ell} = \gbar^{-k} \otimes g^{-k} \cdot
  v_\ell\}$ as our basis for $\C[\Zm] \otimes_\C V$ and order them by
  $v_{k,\ell} < v_{k',\ell'}$ if $k < k'$ or both $k = k'$ and $\ell <
  \ell'$.  Then with the ordered bases used in
  Section~\ref{section:twitorsion}, the map $f$ in (\ref{eq:messymap})
  is an isomorphism of \emph{based} chain complexes.  In particular,
  the complexes have the same torsion, which proves the lemma.
\end{proof}

\begin{proof}[Proof of Lemma~\ref{lem:toreval}]
  Since for any $a \in \Z[\pi]$ we have $(\alpha \otimes \phi)(a)(\xi)
  = \alpha \otimes \lambda_\xi(a)$ the result should follow by
  computing both sides of (\ref{equ:toreval}) with
  Proposition~\ref{prop:fox2}.  The only issue is that we need to
  ensure the nonvanishing of the denominators in (\ref{equ:comptau2})
  for both $\alpha \otimes \phi$ and $\alpha \otimes \lambda_\xi$.
  Since $\a$ is irreducible, we can choose $g \in [\pi, \pi]$ so that
  $\tr(\alpha(g)) \neq 2$ (see e.g.~\cite[Lemma 1.5.1]{CS83}). Notice
  then that $\alpha \otimes \phi ( g^{-1} - 1 ) = \alpha \otimes
  \lambda_\xi ( g^{-1} - 1 ) = \alpha(g^{-1} - 1)$, and since
  $\tr(\alpha(g)) \neq 2$ we have $\det\left(\alpha(g^{-1} - 1)
  \right) \neq 0$.  Hence if we take a suitable presentation of $\pi$
  where $g$ is a generator, then we can apply
  Proposition~\ref{prop:fox2} with $x_i = g$ to both $\alpha \otimes
  \phi$ and $\alpha \otimes \lambda_\xi$ and so prove the lemma.
\end{proof}

We turn now to the proof of Theorem~\ref{thm:taux}, which says that
typically the torsions $\tau(X_m, \a_m)$ collectively determine
$\tk^\a$ (by Theorem~\ref{thm:tauxm}, the hypothesis that $\tau(X_m,
\a_m) \neq 0$ for all $m$ is equivalent to no root of $\tk^\a$ being a
root of unity).  A polynomial $p$ in $\C[t]$ of degree $d$ is
\emph{palindromic} if $p(t) = t^d p(1/t)$, or equivalently if its
coefficients satisfy $a_k = a_{d - k}$ for $0 \leq k \leq d$.  For any
polynomial $p \in \C[t]$ and $m \in \N$, we denote by $r_m(p)$ the
resultant of $t^m-1$ and $p$, i.e.
\[
r_n(p)= \mathrm{Res}(p, t^m - 1) = (-1)^{m d} \mathrm{Res}(t^m - 1, p)
= (-1)^{m d} \produnity p( \zeta )
\]
where here $d$ is the degree of $p$.  The following theorem was proved
by Fried \cite{Fr88} for $p \in \R[t]$ and generalized by Hillar \cite{Hillar2005a} to
the case of $\C[t]$.

\begin{theorem}\label{thm:fried}
Suppose $p$ and $q$ are palindromic polynomials in $\C[t]$.  If
$r_m(p)=r_m(q) \neq 0$ for all $m \in \N$ then $p = q$.
\end{theorem}
\noindent
Theorem~\ref{thm:taux} now follows easily from
Theorems~\ref{thm:tauxm} and \ref{thm:fried} and the symmetry of
$\tk^\a$ shown in Theorem~\ref{thm:Tsym}.

\begin{remark}
  We just saw that, under mild assumptions, the torsions $\tau(X_m,
  \a_m)$ of cyclic covers determine the $\C(t)$-valued torsion
  polynomial $\tk^a$.  It would be very interesting if one could
  directly read off the degree and the top coefficient of $\tk^\a$
  from the $\tau(X_m, \a_m)$.  See \cite{HillarLevine2007} for some of
  what's known about recovering a palindromic polynomial $p$ from the
  sequence $r_m(p)$; in particular, when $p$ is monic and of even
  degree $d$, Sturmfels and Zworski conjecture that one only needs to
  know $r_m(p)$ for $m \leq d/2 + 1$ to recover $p$.
\end{remark}

\section{Torsion polynomials of hyperbolic knots} \label{section:hypknots}\label{section:hyp}

Let $K$ be a hyperbolic knot in an oriented $\Ztwo$-homology sphere $Y$. In this
section, we define the hyperbolic torsion polynomial $\tk$ associated
to a certain preferred lift to $\SL$ of the holonomy representation of
its hyperbolic structure.

\subsection{The discrete and faithful \texorpdfstring{$\SL$}{SL(2,C)} representations}

As usual, we write $\pi=\pi_K\assign\pi_1(X_K)$, and let
$\mu \in \pi$ be a meridian for $K$.  The orientation of $\mu$, or
equivalently of $K$, will not matter in this section, but fix one so
that $\phi \colon \pi \to \Z$ is determined.

From now on assume that $M = Y \setminus K \cong \interior(X)$ has a
complete hyperbolic structure.  The manifold $M$ inherits an
orientation from $Y$, and so its universal cover $\widetilde{M}$ can
be identified with $\H^3$ by an \emph{orientation preserving}
isometry.  This identification is unique up to the action of
$\Isom^+(\H^3) = \PSL$, and the action of $\pi$ on $\widetilde{M} =
\H^3$ gives the \emph{holonomy representation} $\alphabar \maps \pi \to
\PSL$, which is unique up to conjugation.

\begin{remark}
  By Mostow-Prasad rigidity, the complete hyperbolic structure on $M$
  is unique.  Thus $\alphabar$ is determined, up to conjugacy, solely
  by the knot $K$ and the orientation of the ambient manifold $Y$.  A subtle point
  is that there are actually \emph{two} conjugacy classes of discrete
  faithful representations $\pi_K \to \PSL$; the other one corresponds
  to reversing the orientation of $Y$ (not $K$) or equivalently
  complex-conjugating the entries of the image matrices.
\end{remark}

To define a torsion polynomial, we want a representation into $\SL$
rather than $\PSL$.  Thurston proved that $\alphabar$ always lifts to
a representation $\alpha \maps \pi \to \SL$; see \cite{Th97} and
\cite[Section~1.6]{Sh02} for details.  In fact, there are exactly two
such lifts, the other being $g \mapsto (-1)^{\phi(g)}\alpha(g)$; the
point is that any other lift has the form $g \mapsto
\epsilon(g)\alpha(g)$ for some homomorphism $\epsilon \maps \pi \to
\{\pm 1\}$, i.e.~some element of $H^1(M; \Ztwo) = \Ztwo$.  Now
$\alphabar(\mu)$ is parabolic, and so $\tr\big(\alpha(\mu)\big) = \pm 2$.
Since $Y$ is a \ZtwoHS, we know $\phi(\mu)$ is odd; hence there is a
lift $\a$ where $\tr\big(\a(\mu)\big) = 2$; arbitrarily, we focus on
that lift and call it the \emph{distinguished representation}.  This
representation is determined, up to conjugacy, solely by $K$ (sans
orientation). We explain below the simple change that results if we
instead required the trace to be $-2$.

\subsection{The hyperbolic torsion polynomial}\label{sec:hypertorsion}

For a hyperbolic knot $K$ in an oriented \ZtwoHS, we define the
\emph{hyperbolic torsion polynomial} to be
\[ \T_K(t)\assign\T_K^{\a}(t)\] where $\a \maps \pi \to \SL$ is the
distinguished representation.  Before proving Theorem \ref{thm:deft}
which summarizes basic properties of $\T_K(t)$, we give a few
definitions.  The \emph{trace field} $\F_K$ of $K$ is the field
obtained by adjoining to $\Q$ the elements $\tr\big(\alpha(g)\big)$
for all $g \in \pi$; this is a finite extension of $\Q$ and an
important number theoretic invariant of the hyperbolic structure on
$M$; see \cite{MR03} for more.  We say that $K$ has \emph{integral
  traces} if every $\tr\big(\alpha(g)\big)$ is an algebraic integer
(this is necessarily the case if $M$ does not contain a closed
essential surface, see e.g.~\cite[Theorem 5.2.2]{MR03}).  Also, we
denote by $K^*$ the result of switching the orientation of the ambient
manifold $Y$; we call $K^*$ the \emph{mirror image} of $K$.  We call
$K$ \emph{amphichiral} if $Y$ has an orientation reversing
self-homeomorphism which takes $K$ to itself; equivalently, $K = K^*$
in the category of knots in oriented \3-manifolds.
\begin{maintheorem}\label{thm:deft2}
  Let $K$ be a hyperbolic knot in an oriented $\Ztwo$-homology \3-sphere. Then
  $\tk$ has the following properties:
  \begin{enumerate}
  \item \label{item:defsym}
   $\tk$ is an unambiguous element of \hspace{0.01ex} $\ct$ which
    satisfies $\tk(t^{-1}) = \tkt$. It does not depend on an
    orientation of $K$.
  \item \label{item:NT}
    The coefficients of $\tk$ lie in the trace field of $K$. If $K$
    has integral traces, the coefficients of $\tk$ are algebraic integers.
  \item \label{item:nonvanish}
    $\tk(\xi)$ is non-zero for any root of unity $\xi$.
  \item \label{item:mirror}
    If $K^*$ denotes the mirror image of $K$, then
    $\T_{K^*}(t)=\oltk(t)$, where the coefficients of the latter
    polynomial are the complex conjugates of those of $\T_K$.
  \item \label{item:amph}
    If $K$ is amphichiral then $\tk$ is a real polynomial.
  \item \label{item:mut} The values $\tk(1)$ and $\tk(-1)$ are mutation invariant.
\end{enumerate}
\end{maintheorem}
\noindent
For the special case of 2-bridge knots and $\xi=\pm 1$, the assertion
(\ref{item:nonvanish}) is also a consequence of the work of
Hirasawa-Murasugi \cite{HM08} and Silver-Williams \cite{SW09c}.
\begin{proof}
  Since the distinguished representation $\a$ is irreducible, part
  (\ref{item:defsym}) follows from Theorems~\ref{thm:Tsym} and
  \ref{thm:ktm05}(\ref{item:irred}).

  Next, since $M$ has a cusp, by Lemma~2.6 of \cite{NeumannReid1992}
  we can conjugate $\a$ so that its image lies in $\SLFK$, where
  $\F_K$ is the trace field; hence $\tk \in \F_K\tpm$ proving the
  first part of (\ref{item:NT}).  To see the other part, first using
  \cite[Theorem~5.2.4]{MR03} we can conjugate $\a$ so that $\a(\pi)
  \subset \SLOK$, where here $\OO_\K$ is the ring of algebraic
  integers in some number field $\K$ (which might be a proper
  extension of $\F_K$).  We now compute $\tk^\a$ by applying
  Proposition~\ref{prop:fox} to a presentation of $\pi$ where $\mu$ is
  our preferred generator.  Since $\a(\mu)$ is parabolic with trace 2,
  the denominator in (\ref{equ:comptau}) is $p(t) \assign \det\big(
  (\a \otimes \phi)( \mu^{-1} - 1 )\big) = (t^k - 1)^2$ where $k =
  -\phi(\mu) \neq 0$.  Thus by (\ref{equ:comptau}), we know $p(t)
  \cdot \tk^\a$ is in $\OO_K\tpm$.  Then since $p(t) \in \Z[t^{\pm
    1}]$ is monic, the lead coefficient of $\tk^\a$ must be integral.
  An easy inductive argument now shows that all the other coefficients
  are also integral, proving part (\ref{item:NT}).

  The proof of (\ref{item:nonvanish}) uses Theorem~\ref{thm:tauxm},
  and we handle all \mth\ roots of unity at once.  In
  the notation of Section~\ref{sec:cyclic}, we have
  \begin{equation}\label{eq:prodrepeat}
    \produnity \T_K(\zeta)= \tau(X_m,\a_m).
  \end{equation}
  By Menal--Ferrer and Porti \cite[Theorem~0.4]{MP09}, which builds on
  work of Raghunathan \cite{Ra65}, we have that
  $H_*^{\a_m}(X_m;\C^2)=0$, or equivalently, $\tau(X_m,\a_m)$ is
  non-zero.  Thus by (\ref{eq:prodrepeat}) we must have $\T_K(\zeta)
  \neq 0$ for any \mth root of unity, establishing part
  (\ref{item:nonvanish}).

  For (\ref{item:mirror}), the distinguished representation for the
  mirror knot $K^*$ is $\alphabar \maps \pi \to \SL$ where each
  $\alphabar(g)$ is the matrix which is the complex conjugate of
  $\alpha(g)$. Since our choice of orientation for the meridian $\mu$
  was arbitrary, we can use the same $\phi$ for when calculating both $\tk$ and
  $\T_{K^*}$.  Thus we have
  \[
  \T_{K^*}(t) = t^{\phi(c_1(\ss))}\tau(X, \alphabar \otimes \phi) =
  t^{\phi(c_1(\ss))} \ol{\tau(X_k, \alpha \otimes \phi)}
  = \oltk(t)
  \]
  proving (\ref{item:mirror}).  Next, claim (\ref{item:amph}) follows immediately
  from (\ref{item:mirror}).  Finally, claim (\ref{item:mut}) is a
  recent result of Menal-Ferrer and Porti \cite{MenalFerrerPorti2011}.
\end{proof}

As in Section~\ref{sec:cyclic}, we now consider the $\C$-valued
torsions $\tau(X_m, \a_m)$ of finite cyclic covers of $X_K$.
Somewhat surprisingly, these determine $\tk$:
\begin{theorem}\label{prop:tktdetermined}
  Let $K$ be a hyperbolic knot in a \ZtwoHS\ with distinguished
  representation $\a \maps \pi_K \to \SL$. Then $\tk$ is determined by
  the torsions $\tau(X_m,\a_m)\in \C$.
\end{theorem}

\begin{proof}
  As discussed in the proof of
  Theorem~\ref{thm:deft}(\ref{item:nonvanish}), every $\tau(X_m, \a_m)
  \neq 0$, so the result is immediate from Theorem~\ref{thm:taux}.
\end{proof}

\begin{remark}
  When choosing our distinguished representation, we arbitrarily chose
  the lift $\alpha \maps \pi \to \SL$ where $\tr(\alpha(\mu)) = 2$.  As
  discussed, the other lift $\beta$ is given by $g \mapsto (-1)^{\phi(g)}
  \alpha(g)$.  Note that given $g\in \pi$ we have
  \begin{align*}\big((\b\otimes \phi)(g)\big)(t) =\b(g)\cdot
    t^{\phi(g) }&=\a(g)\cdot(-1)^{\phi(g)}\cdot t^{\phi(g)} \\
  &=\a(g)\cdot (-t)^{\phi(g)}=\big((\a\otimes \phi)(g)\big)(-t).
  \end{align*}
  It follows from Proposition \ref{prop:fox} that $\T_K\myupbeta(t)=\T_K^\a(-t)=\T_K(-t)$.
  Put differently,  using $\beta$ instead of $\alpha$ would simply replace $t$ by $-t$.
\end{remark}

\begin{remark}
  When $Y$ is not a \ZtwoHS, the choice of lift $\a$ of the holonomy
  representation can have a more dramatic effect on $\T^\a$.  For
  example, consider the manifold $m130$ in the notation of
  \cite{CallahanHildebrandWeeks1999,SnapPy}.  This manifold is a
  twice-punctured genus 1 surface bundle over the circle, and as
  $H_1(M; \Z) = \Z\oplus\Z_8$, there are $4$ distinct lifts of the
  holonomy representation.  Two of these lifts give $\tk^\a = \Tsym{2}
  - 2 i$ and the other two give $\tk^\a = \Tsym{2} + \sqrt{-8 -
    8i}\Tsym{1} - 6 i$ for the two distinct square-roots of $-8-8i$.
  In particular, the fields generated by the coefficients are
  different; only the latter two give the whole trace field.
\end{remark}

\section{Example: The Conway and Kinoshita--Terasaka knots}
\label{section:ktc}\label{section:examples}

The Conway and  Kinoshita-Terasaka knots are a famous pair of
mutant knots which both have trivial Alexander polynomial.   Despite
their close relationship, they have different genera.  Thus they are a
natural place to start our exploration of $\tk$, and we devote this
section to examining them in detail.

The Conway knot $C$ is the mirror of the knot $11n34$ in the numbering
of \cite{Knotscape, HTW98}.  The program Snap \cite{Snap, SnapPaper}
finds that the trace field $\F$ of the hyperbolic structure on the
exterior of $C$ is the extension of $\Q$ gotten by adjoining the root $\theta
\approx 0.1233737 - 0.5213097 i$ of
\[
p(x) = x^{11} - x^{10} + 3 x^9 - 4 x^8 + 5 x^7 - 8 x^6 + 8 x^5 - 5 x^4 + 6 x^3 - 5 x^2 + 2 x - 1.
\]
Snap also finds the explicit holonomy representation $\pi_C \to
\mathrm{SL}(2, \F)$, and one can directly apply
Proposition~\ref{prop:fox} to compute $\T_{C}$.  If we set
\[
\eta = \frac{1}{53}\big(20 \theta^{10} + 9 \theta^{9} + 28\theta^{8} + 3
\theta^{7} + \theta^{6} + 19 \theta^{5} + 10 \theta^{4} + 47
\theta^{3} + 6 \theta + 1\big)
\]
then $\left\{ \eta, \theta, \theta^2, \ldots, \theta^{10}\right\}$ is an integral
basis for $\mathcal{O}_{\F}$, and we find
\begin{align*}
\smalleqsize \T_{C}(t) \ =& \ \smalleqsize \left( -79 \theta^{10} - 35 \theta^{9} - 111 \theta^{8} - 11 \theta^{7} - 4 \theta^{6} - 71 \theta^{5} - 38 \theta^{4} - 187 \theta^{3} - 2 \theta^{2} - 24 \theta + 206 \eta\right) \Tsym{5}  \\
\messylinestart + \left( 257 \theta^{10} + 114 \theta^{9} + 361
  \theta^{8} + 36 \theta^{7} + 13 \theta^{6} + 232 \theta^{5} + 124
  \theta^{4} + 608 \theta^{3} + 6 \theta^{2} + 78 \theta - 671
  \eta\right) \Tsym{4}  \\
\messylinestart + \left( -372 \theta^{10} - 165 \theta^{9} - 523 \theta^{8} - 51 \theta^{7} - 21 \theta^{6} - 334 \theta^{5} - 183 \theta^{4} - 877 \theta^{3} - 11 \theta^{2} - 111 \theta + 972 \eta\right) \Tsym{3}  \\
\messylinestart+ \left( 373 \theta^{10} + 162 \theta^{9} + 528 \theta^{8} + 40 \theta^{7} + 33 \theta^{6} + 312 \theta^{5} + 200 \theta^{4} + 866 \theta^{3} + 24 \theta^{2} + 99 \theta - 968 \eta\right) \Tsym{2}  \\
\messylinestart + \left( -303 \theta^{10} - 115 \theta^{9} - 445 \theta^{8} + 14 \theta^{7} - 75 \theta^{6} - 152 \theta^{5} - 227 \theta^{4} - 649 \theta^{3} - 73 \theta^{2} - 29 \theta + 749 \eta\right) \Tsym{1}  \\
\messylinestart \smalleqsize+ \left( 116 \theta^{10} + 14 \theta^{9} + 200 \theta^{8} - 88 \theta^{7} + 116 \theta^{6} - 122 \theta^{5} + 204 \theta^{4} + 146 \theta^{3} + 124 \theta^{2} - 78 \theta - 220 \eta\right)\\
\smalleqsize \approx & \smalleqsize \  \left(4.89524+0.09920i\mvph\right)
\Tsym{5} + \left(-15.68571-0.29761i\mvph \right)\Tsym{4} +
\left(23.10363-0.07842i\mvph\right) \Tsym{3} \\
\messylinestart + \left(-26.94164+4.84509i\mvph\right) \Tsym{2} + \left(38.38349-24.49426i\mvph\right) \Tsym{1} + \left(-43.32401+44.08061i\mvph\right).
\end{align*}

The Kinoshita--Terasaka knot is the mirror of $11n42$.  Its trace field is the same
as for the Conway knot (since $[\F : \Q]$ is odd, the trace field is
also the invariant trace field, which is mutation invariant), and one finds
\begin{align*}
\smalleqsize \T_{\KT}(t) \ =& \ \smalleqsize \left( -55 \theta^{10} - 24 \theta^{9} - 78 \theta^{8} - 6 \theta^{7} - 5 \theta^{6} - 45 \theta^{5} - 29 \theta^{4} - 128 \theta^{3} - 5 \theta^{2} - 15 \theta + 142 \eta\right) \Tsym{3}  \\
\messylinestart + \left( 293 \theta^{10} + 126 \theta^{9} + 416 \theta^{8} + 28 \theta^{7} + 29 \theta^{6} + 236 \theta^{5} + 160 \theta^{4} + 678 \theta^{3} + 24 \theta^{2} + 75 \theta - 756 \eta\right) \Tsym{2}  \\
\messylinestart + \left( -699 \theta^{10} - 291 \theta^{9} - 1001 \theta^{8} - 42 \theta^{7} - 95 \theta^{6} - 512 \theta^{5} - 419 \theta^{4} - 1585 \theta^{3} - 81 \theta^{2} - 149 \theta + 1785 \eta\right) \Tsym{1}  \\
\messylinestart + \left( 790 \theta^{10} + 314 \theta^{9} + 1146 \theta^{8} + 8 \theta^{7} + 150 \theta^{6} + 494 \theta^{5} + 532 \theta^{4} + 1738 \theta^{3} + 136 \theta^{2} + 126 \theta - 1986 \eta\right)  \\
\approx& \   \smalleqsize \left(4.41793-0.37603i\mvph\right) \Tsym{3}
+ \left(-22.94164+4.84509i\mvph\right) \Tsym{2} +
\left(61.96443-24.09744i\mvph\right) \Tsym{1} \\
\messylinestart + \left(-82.69542+43.48539i\mvph\right).
\end{align*}
From the above we see that $\tk$ is not invariant under mutation.
Since $C$ and $\KT$ have genus 3 and 2 respectively and $\deg (\T_{C})
= 10$ and $\deg( \T_{\KT}) = 6$, we see that Conjecture \ref{conj:top}
holds for both knots.  Also note that the coefficients of these
polynomials are not real, certifying the fact that both knots are
chiral.

\begin{remark}\label{remark:Freps}
  It was shown in \cite[Section~5]{FK06} that twisted Alexander
  polynomials corresponding to representations over finite fields
  detect the genus of all knots with at most twelve crossings.  For
  example, for the Conway knot there is a representation $\a\colon
  \pi_1(X_{C})\to \gl(4,\F_{13})$ such that the corresponding torsion
  polynomial $\T_C^\a \in \F_{13}\tpm$ has degree 14, and hence
\[ x(C) \geq \frac{1}{4}\deg\left( \T_C^\a \right) = 3.5. \]
In particular this shows that $x(C) = 5$ since $x(C) =
2 \genus(C)  - 1 $ is an odd integer.

The calculation using the discrete and faithful $\SL$ representation
is arguably more satisfactory since it gives the equality
\[
x(C) = \frac{1}{2} \deg(\T_C)
\]
on the nose, and not just after rounding up to odd
integers. Interestingly, we have not found an example where this
rounding trick applies to $\tk$; at least for knots with at most 15
crossings one always has $x(K) = \deg(\tk)/2$ (see
Section~\ref{sec:15cross}).
\end{remark}

\subsection{The adjoint representation}\label{section:adj}
For an oriented hyperbolic knot $K$ with distinguished representation
$\a\colon \pi_1(X_K)\to \SL$, we now consider the adjoint representation
\[ \ba{rcl} \a_{\adj}\colon \pi_1(X_K)&\to& \Aut\big(\sltwoc\big)\\
g&\mapsto & A\mapsto \a(g) A \, \a(g)^{-1}\ea \] associated to $\a$.
It is well-known that this representation is also faithful and
irreducible.  Using sign-refined torsion and the orientation on $K$, one
gets an invariant $\tadjk \in \ct$ which is well-defined up to
multiplication by an element of the form $t^k$.  We refer to
\cite{DY09} for details on this construction and for further
information on $\tadjk$; one thing they show is that $\tadjkt =
- \tadjk(t^{-1})$ up a power of $t$, and so $\tadjk$ has odd degree.

For the Conway knot we calculate that
\begin{align*}
\smalleqsize \T_{C}^{\adj}(t) \ \approx& \ \smalleqsize \left(-0.2788 + 16.4072i\mvph\right)\left(t^{13} - 1\right) + \left(-3.9858 - 20.1706i\mvph\right) \left(t^{12} - t\right)
+ \left(-4.2204 - 60.5497i\mvph\right) \left(t^{11} - t^2\right) \\
\messylinestart + \left(52.0953 + 134.5013i\mvph\right) \left(t^{10} - t^3\right)
+ \left(-147.7856 - 46.07448i\mvph\right)\left(t^9 - t^4\right)
+ \left(897.2087+ 62.3265i\mvph\right) \left(t^8 - t^5 \right) \\
\messylinestart + \left(-2465.8556 - 1308.0110i\mvph\right) \left( t^7 - t^6 \right)
\end{align*}
and for the Kinoshita-Terasaka knot we found
\begin{align*}
\smalleqsize \T_{\KT}^{\adj}(t) \ \approx& \ \smalleqsize
\left(-0.7378 + 12.4047i\mvph\right)\left(t^7 - 1\right)
+ \left(29.9408 - 56.5548i\mvph\right)\left(t^6 - t\right) \\
\messylinestart + \left(-655.7823 - 173.0400i\mvph\right)\left(t^5 - t^2\right)
+ \left(2056.7509 + 1678.4875i\mvph\right)\left(t^4 - t^3\right).
\end{align*}
As $\dim \big(\sltwoc\big)=3$, it follows from Theorem \cite[Theorem~1.1]{FK06} that
\begin{equation}\label{eq:CKTadj}
x(K) \geq \frac{1}{3}\deg\left(\tadjkt\right) \mtext{and hence} x(C)
\geq \frac{13}{3}   \mtext{and} x(\KT) \geq \frac{7}{3}.
\end{equation}
Thus using that $x(K)$ is an integer, we get $x(C) \geq 5$ and
$x(\KT) \geq 3$, which are sharp.  Intriguingly, unlike for $\tk$ one
does \emph{not} have equality in (\ref{eq:CKTadj}) for these two
knots.  Below in Section~\ref{sec:adjointdata}, we describe some knots
where $\tk^\adj$ fails to give a sharp bound on $x(K)$ even after
using that $x(K)$ is an odd integer. 

\section{Knots with at most 15 crossings}
\label{sec:15cross}

There are $313{,}231$ prime knots with 15 or fewer crossings
\cite{HTW98}, of which all but 22 are hyperbolic.  For each of these
hyperbolic knots, we computed a high-precision numerical approximation
to $\tk$ (see Section~\ref{sec-comp-details} for details), and this
section is devoted to describing the various properties and patterns
we found.

\subsection{Genus}

The genus bound from $\tk$ given in Theorem~\ref{thm:top} is
sharp for all $313{,}209$ hyperbolic knots with 15 or fewer crossings;
that is, $x(K) = \deg(\tk)/2$ for all these knots.  In contrast, the
ordinary Alexander polynomial fails to detect the genus for $8{,}834$
of these knots, which is $2.8\%$ of the total.

We showed the genus bound from $\tk$ was sharp using the following
techniques to give upper bounds on the genus.  First, for the
alternating knots ($36\%$ of the total), the genus is simply
determined by the Alexander polynomial \cite{Murasugi1958,
  Crowell1959}. For the nonalternating knots, we first did $0$-surgery
on the knot $K$ to get a closed 3-manifold $N$; by Gabai
\cite{Gabai1987}, the genus of $K$ is the same as that of the simplest
homologically nontrivial surface in $N$.  We then applied the method
of Section 6.7 of \cite{DR10} to a triangulation of $N$ to quickly
find a homologically nontrivial surface.  As this surface need not be
minimal genus, when necessary we randomized the triangulation of $N$
until we found a surface whose genus matched the lower bound from
$\tk$.

\subsection{Fibering}

We also found that $\tk$ gives a sharp obstruction to fibering for all
hyperbolic  knots  with at  most  15  crossings.   In particular,  the
$118{,}252$ hyperbolic knots where $\tk$ is monic are all fibered.  In
contrast,  while the  ordinary Alexander  polynomial  always certifies
nonfibering for  alternating knots \cite{Murasugi1963,  Gabai1986},
among the $201{,}702$ nonalternating knots there are $7{,}972$  or
$4.0\%$ whose Alexander polynomials are monic but don't fiber.

To confirm fibering when $\tk$ was monic, we used a slight
generalization of the method of Section 6.11 of \cite{DR10}.  Again by
\cite{Gabai1987}, it is equivalent to show that the 0-surgery $N$ is
fibered.  Starting with the minimal genus surface $S$ found as above, we
split $N$ open along $S$, and tried to simplify a presentation for the
fundamental group of $N \setminus S$ until it was obviously that of a
surface group.  If it is, then it follows that $N \setminus S = S
\times I$ and $N$ is fibered.  The difference with \cite{DR10} is that
we allowed $S$ to be a general normal surface instead of the
restricted class of Figure~6.13 of \cite{DR10}.  We handled this by
splitting the manifold open along $S$ and triangulating the result
using Regina \cite{Regina}.

\subsection{Chirality}
\label{sec:chiral}

For hyperbolic knots with at most 15 crossings, we found that a knot
was amphichiral if and only if $\tk$ had real coefficients.  In
particular, there are 353 such knots with $\tk$ real, and SnapPy
\cite{SnapPy} easily confirms that they are all amphichiral.
(This matches the count of amphichiral knots from Table A1 of
\cite{HTW98}.)

In contrast, the numbers $\tk(1)$ and $\tk(-1)$ do not always detect
chirality.  For example, the chiral knot $10_{153} = 10n10$ has
$\tk(1) = 4$ and $10_{157} = 10n42$ has $\tk(-1) = 576$.  Moreover,
the knot $14a506$ has both $\tk(1)$ and $\tk(-1)$ real.  (This last
claim was checked to the higher precision of 10{,}000 decimal places.)

\subsection{Knots with the same $\tk$}
\label{sec:ex-mutants}
While we saw in Section~\ref{section:ktc} that $\tk$ is \emph{not}
mutation invariant, there are still pairs of knots with the same
$\tk$.  In particular, among knots with at most 15 crossings, there
are $2{,}739$ groups of more than one knot that share the same $\tk$,
namely $2{,}700$ pairs and $39$ triples.  Here, we do not distinguish
between a knot and its mirror image, and having the same $\tk$ means
that the coefficients agree to $5{,}000$ decimal places.  Stoimenow
found there are $34{,}349$ groups of mutant knots among those with
at most 15 crossings, involving some $77{,}680$ distinct knots
\cite{St10}.  Thus there are many examples where mutation changes
$\tk$.  However, all of the examples we found of knots with the same
$\tk$ are in fact mutants.

As mentioned, Menal-Ferrer and Porti \cite{MenalFerrerPorti2011}
showed that the evaluations $\tk(1)$ and $\tk(-1)$ are mutation
invariant.  We found 38 pairs of non-mutant knots with the same
$\tk(1)$ and the same $\tk(-1)$.  Suggestively, several of these pairs
(including the five pairs shown in \cite[Figure~3.9]{DGST2010}, see
also \cite[Tables~2 and~3]{StTa09}) are known to be genus-2 mutants.
We also found a triple of mutually non-mutant knots $\{ 10a121,
12a1202, 12n706\}$ where $\tk(+1) = -4$, and a similar sextet
$\{10n10, 12n881, 13n592, 13n2126, 15n9378, 15n22014\}$ where $\tk(+1)
= 4$; however, within these groups, the value $\tk(-1)$ did not agree.

\subsection{Other patterns}
\label{sec:otherpat}

We found two other intriguing patterns which we are unable to explain.
The first is that the second highest coefficient of $\tk$ is often
real for fibered knots.  In particular, this is the case for $53.1\%$
(62{,}763 of 118{,}252) of the fibered knots in this sample.  In
contrast, the second coefficient is real for only $0.2\%$ (364 of
194{,}957) of nonfibered knots.  (Arguably, the right comparison is
with the lead coefficient of $\tk$ for nonfibered knots; even fewer
($0.05\%$) of these are real.)  For fibered knots, the twisted
homology of the universal cyclic cover can be identified with that of
the fiber; hence the action of a generator of the deck group on this
homology of the cover can be thought of as the action of the monodromy
of the bundle on the twisted homology of the fiber.  The second
coefficient of $\tk$ is then just the sum of the eigenvalues of this
monodromy, but it's unclear why this should often be a real number.

The second observation is that $|\tk(-1)|>|\tk(1)|$ for all but 22 ($<
0.01\%$) of these knots.   The exceptions are nonalternating and
all but one ($15n151121$) is fibered.

\subsection{Adjoint polynomial}
\label{sec:adjointdata}

As discussed in Section~\ref{section:adj}, Dubois and Yamaguchi
\cite{DY09} studied a torsion polynomial $\tk^\adj$ constructed by
composing the holonomy representation with the adjoint representation
of $\PSL$ on its Lie algebra.  We also numerically calculated this
invariant for all knots with at most 15 crossings.  Unlike what we
found for $\tk$, there was not always an equality in the bound
(\ref{eq:CKTadj}).  In fact, some 8{,}252 of these knots had $x(K) >
(1/3)\deg( \tk^\adj)$.  All such knots were non-alternating, and were
among the 8{,}834 knots where $\Delta_K$ fails to give a sharp bound
on $x(K)$.  However, using the trick from Section~\ref{remark:Freps}
that $x(K)$ is an odd integer, the bound on $x(K)$ from $\tk^\adj$ was
effectively sharp in all but 12 cases.  The 12 knots for which
$\tk^\adj$ fails to determine the genus are as follows: there are 7
knots where $x(K) = 7$ (i.e.~genus 4) but $\deg(\tk^\adj) = 15$,
namely 
\[ 
\{15n75595, 15n75615, 15n75858, 15n75883, 15n75948, 15n99458,
15n112466 \}
\]
 and 5 knots with $x(K) = 9$ (i.e.~genus 5) but
$\deg(\tk^\adj) = 21$, namely 
\[ 
\{15n59545, 15n62671, 15n68947, 15n109077, 15n85615\}.
\]
In these 12 cases, we computed $\tk^\adj$ to the higher accuracy of
10{,}000 decimal places.

Intriguingly, the polynomial $\tk^\adj$ did better at providing an
obstruction to fibering; just as for $\tk$, it was monic only for
those knots in the sample that are actually fibered.

\subsection{Computational details}\label{sec-comp-details}

The complete software used for these computations, as well as a table
of $\T_K$ for all these knots, is available at \cite{software}.  The
software runs within Sage \cite{SAGE}, and makes use of SnapPy
\cite{SnapPy} and t3m \cite{t3m}.  It finds very high-precision
solutions to the gluing equations, in the manner of Snap \cite{Snap,
  SnapPaper}, and extracts from this a high-precision approximation to
the distinguished representation.  Except as noted above, we did all
computations with 250 decimal places of precision.  Even at this
accuracy, $\tk$ is fast to compute for these knots, taking only a
couple of seconds each on a late 2010 high-end desktop computer.
However, to save space, only 40 digits were saved in the final table.

To guard against error, two of the authors independently wrote
programs which computed $\T_K$, and the output of these programs were
then compared for all nonalternating knots with 14 crossings.

\section{Twisted torsion and the character variety of a knot}  \label{section:charactervariety}

As usual, consider a hyperbolic knot $K$ in a \ZtwoHS, and let $\pi
\assign \pi_1(X_K)$.  So far, we have focused on the torsion
polynomial of the distinguished representation $\alpha \maps \pi \to
\SL$ coming from the hyperbolic structure.  However, this
representation is always part of a complex curve of representations
$\pi \to \SL$, and it is natural to ask if there is additional
topological information in the torsion polynomials of these other
representations.  In this section, we describe how to understand all
of these torsion polynomials at once, and use this to help explain
some of the patterns observed in Section~\ref{sec:15cross}.  For the
special case of 2-bridge knots, Kim and Morifuji \cite{Mo08,KM10} had
previously studied how the torsion polynomial varies with the
representation, and we extend here some of their results to
more general knots.

To state our results, we must first review some basics about character
varieties; throughout, see the classic paper \cite{CS83} or the survey
\cite{Sh02} for details.  Consider the \emph{representation variety}
$R(K)\assign\hom\big(\pi,\SL\big)$ which is an affine algebraic
variety over $\C$.  The group $\SL$ acts on $R(K)$ by conjugating each
representation; the algebro-geometric quotient $X(K) \assign
R(K)/\!/\,\SL$ is called the \emph{character variety}.  More
concretely, $X(K)$ is the set of characters of representations $\alpha
\in R(K)$, i.e.~functions $\chi_\a \maps \pi \to \C$  of the form
$\chi_\a(g) = \tr\left(\a(g)\right)$ for $g \in \pi$.  When $\alpha$
is irreducible, the preimage of $\chi_\a$ under the projection $R(K)
\to X(K)$ is just all conjugates of $\alpha$, but distinct conjugacy
classes of reducible representations can sometimes have the same
character.  Still, it makes sense to call a character irreducible or
reducible depending on which kind of representations it comes from.

The character variety $X(K)$ is also an affine algebraic variety over
$\C$; its coordinate ring $\C[X(K)]$, which consists of all regular
functions on $X(K)$, is simply the subring $\C[R(K)]^\SL$ of regular
functions on $R(K)$ which are invariant under conjugation. We start by
showing that it makes sense to define a torsion polynomial $\tk^\chi$
for $\chi \in X(K)$ via $\tk^\chi \assign \tk^\a$ for any $\a$ with
$\chi_\a = \chi$.
\begin{lemma}\label{lem:chitor}
  If $\a, \b \in R(K)$ have the same character, then $\tk^\a = \tk\myupbeta$.
\end{lemma}
\noindent

\begin{proof}[Proof of Lemma~\ref{lem:chitor}]
  As discussed, if $\a$ is irreducible then $\b$ must be conjugate to
  $\a$; hence they have the same torsion polynomial.  If instead $\a$ is
  reducible, then Theorem~\ref{thm:ktm05}(\ref{item:redone}) shows
  that $\tk^\a$ depends only on the diagonal part of $\a$, which can
  be recovered from its character.  Since $\b$ must also be reducible
  and has the same character as $\a$, we again get $\tk^\a = \tk\myupbeta$.
\end{proof}

An irreducible component $X_0$ of $X(K)$ has $\dim_\C(X_0) \geq 1$
since the exterior of $K$ has boundary a torus.  There are two
possibilities for $X_0$: either it consists entirely of reducible
characters, or it contains an irreducible character.  In the latter
case, it turns out that irreducible characters are Zariski open in
$X_0$, and \emph{every} character in $X_0$ is that of a representation
with non-abelian image. As the torsion polynomials of reducible
representations are boring (see Theorem~\ref{thm:ktm05} and the
discussion immediately following), we focus on those components
containing an irreducible character.  We denote the union of all such
components as $X(K)\nab$; equivalently, $X(K)\nab$ is the Zariski
closure of the set of irreducible characters.

It is natural to ask how $\tk^\chi$ varies as a function of $\chi$.
We find:
\begin{chartheorem}
  Let $X_0$ be an irreducible component of $X(K)\nab$.  There is a
  unique $\tkx \in \cxt$ so that for all $\chi \in X_0$ one has $
  \tk^\chi(t) = \tkx(\chi)(t)$.  Moreover, $\tkx$ is itself the
  torsion polynomial of a certain representation $\pi \to \SLF$ and thus
  has all the usual properties (symmetry, genus bound, etc.).
\end{chartheorem}
\noindent
We give several explicit examples of $\tkx$ later in
Section~\ref{sec:charexs}. The following result is immediate from
Theorem~\ref{thm:chartorsion}.
\begin{corollary}\label{cor:chartorsion}
 Let $X_0$ be an irreducible component of $X(K)\nab$.  Then
  \begin{enumerate}
    \item For all $\chi \in X_0$, we have $\deg\left( \tk^\chi\right)  \leq \deg
      \left(\tkx\right)$ with equality on a nonempty Zariski open subset.
    \item If $\tkx$ is monic, then $\tk^\chi$ is monic for all $\chi
      \in X_0$.  Otherwise, $\tk^\chi$ is monic only on a proper
      Zariski closed subset.
  \end{enumerate}
\end{corollary}
\noindent
In particular, when $X_0$ is a curve, the genus bound and
fibering obstruction given by $\tk^\chi$ are the same for all $\chi
\in X_0$ except on a finite set where $\tk^\chi$ provides weaker
information.  We can also repackage Corollary~\ref{cor:chartorsion} as
a uniform statement on all of $X(K)$.
\begin{charcorollary}
  Let $K$ be a knot in an integral homology 3-sphere.   Then
  \begin{enumerate}
  \item The set $\setdef{ \chi \in X(K)}{\deg(\tk^\chi ) = 2 x(K)}$ is Zariski open.
  \item The set $\setdef{ \chi \in X(K)}{\tk^\chi \mbox{ is monic }}$ is Zariski closed.
\end{enumerate}
\end{charcorollary}
\begin{proof}
  It suffices to consider the intersections of these sets with each
  irreducible component $X_0$ of $X(M)$.  If $X_0$ consists solely
  of reducible characters, the result is immediate from
  Theorem~\ref{thm:ktm05}(\ref{item:redtwo}).  Otherwise, it
  follows from Corollary~\ref{cor:chartorsion} combined with the fact
  that $\deg(\tkx) \leq 2 x(K)$.
\end{proof}

We now turn to the proof of Theorem~\ref{thm:chartorsion}.

\begin{proof}[Proof of  Theorem~\ref{thm:chartorsion}]
  By Proposition 1.4.4 of \cite{CS83}, there is an irreducible
  component $R_0$ of $R(K)$ where the projection $R_0 \to
  X(K)$ surjects onto $X_0$.  Consider the tautological representation
  \[
  \rhotaut \maps \pi \to \SLCR
  \]
  which sends $g \in \pi$ to the matrix $\rhotaut(g)$ of regular
  functions on $R_0$ so that
  \[
  \rhotaut(g)(\a) = \a(g) \mtext{for all $\a \in  R_0$.}
  \]
  Since $R_0$ is irreducible, $\cR$ is an integral domain.
  Thus we can consider its field of fractions, i.e.~the field of
  rational functions $\cRf$.  Working over $\cRf$ there is an
  associated torsion polynomial $\tktaut$ which is in $\cRf\tpm$ since
  $\rhotaut$ is irreducible.  From Lemma~\ref{prop:fox}, it is clear
  that for every $\a \in R_0$ we have $\tk^\a(t) = \tktaut(\a)(t)$ in
  $\ct$.  Hence the coefficients of $\tktaut$ have well-defined values
  at every point $\a \in R_0$, and so lie in $\cR$.  Now since the
  torsion polynomial is invariant under conjugation, each coefficient of
  $\tktaut$ lies in $\cx = \C[R_0]^\SL$, and hence $\tktaut$ descends
  to an element of $\cxt$, which is the $\tkx$ we seek.
\end{proof}

\subsection{The distinguished component}

It is natural to focus on the component $X_0$ of $X(M)$ which
contains the distinguished representation.  In this case $X_0$ is an
algebraic curve, and we refer to it as the \emph{distinguished
  component}.  By Corollary~\ref{cor:chartorsion}, the following
conjecture that $\tkx$ detects both the genus and fibering of $K$ is
implied by Conjecture~\ref{conj:top}.
\begin{charconjecture}
  Let $K$ be a hyperbolic knot in $S^3$, and $X_0$ be the distinguished
  component of its character variety.  Then $2 x(K) = \deg( \tkx )$
  and $\tkx$ is monic if and only if $K$ is fibered.
\end{charconjecture}
\noindent
As we explain in Section~\ref{sec:2bridge}, this conjecture is true
for many 2-bridge knots.

One pattern in Section~\ref{sec:15cross} is that $\tk$ never gave
worse topological information than the ordinary Alexander
polynomial $\Delta_K$.  In certain circumstances, Corollary~\ref{cor:chartorsion}
allows us to relate $\Delta_K$ to $\tk$ as we now discuss.  First, we
can sometimes show that $\tkx$ must contain at least as much
topological information as $\Delta_K$.
\begin{lemma}\label{lem:consred}
  Let $K$ be a knot in an \ZHS.  Suppose $X_0$ is a component of
  $X(K)\nab$ which contains a reducible character.  Then $\deg(\tkx)
  \geq 2 \deg(\Delta_K) - 2$ and if $\Delta_K$ is nonmonic so is
  $\tkx$.
\end{lemma}
\begin{proof}
  Let $\alpha$ be a reducible representation whose character lies in
  $X_0$.  By Theorem~\ref{thm:ktm05}(\ref{item:redtwo}), the
  torsion polynomial $\tk^\a$ has degree $2 \deg(\Delta_K) - 2$ and its
  lead coefficient is the square of that of $\Delta_K \in \Z \tpm$.
  The result now follows from Corollary~\ref{cor:chartorsion}.
\end{proof}

Now, consider the distinguished representation $\alpha$ and
distinguished component $X_0 \subset X(K)$.  We say that $\alpha$ is
\emph{sufficiently generic} if $\deg(\tk) = \deg(\tkx)$ and $\tk$ is
monic only if $\tkx$ is.  Corollary~\ref{cor:chartorsion} suggests
that most knots will have sufficiently generic distinguished
representations; however, because the distinguished character takes on
only algebraic number values, there seems to be no a~priori reason why
this must always be the case.  Regardless, our intuition is
that the hypothesis of this next proposition holds quite often:
\begin{proposition}
  Let $K$ be a knot in an \ZHS\ whose distinguished representation is
  sufficiently generic, and whose distinguished component of
  $X(M)$ contains a reducible character.  Then
  $\deg(\tk) \geq 2 \deg(\Delta_K) - 2$ and if $\Delta_K$ is nonmonic
  so is $\tk$.
\end{proposition}

\subsection{2-bridge knots} \label{sec:2bridge}

For 2-bridge knots in $S^3$, Kim and Morifuji previously studied the
torsion polynomial as a function on $X(M)\nab$ in \cite{KM10}.  As
2-bridge knots are alternating, the ordinary Alexander polynomial
$\Delta_K$ determines the genus and whether $K$ fibers
\cite{Murasugi1963, Crowell1959, Murasugi1963, Gabai1986}.  However,
as mentioned, there seems to be no a~priori reason that the same must
be true for $\tkx$.  We now sketch what is known about this special
case, starting with two results from \cite{KM10}.
\begin{theorem}[{\cite[Theorem 4.2]{KM10}}]
  Let $K$ be a hyperbolic 2-bridge knot.  Then there exists a component $X_0
  \subset X(K)\nab$ such that $2 x(K)
  = \deg( \tkx )$ and $\tkx$ is monic if and only if $K$ is fibered.
  In particular, if $X(K)\nab$ is irreducible, then
  Conjecture~\ref{conj:char} holds for $K$.
\end{theorem}

\begin{proof}
  Since $\Delta_K$ detects the genus, it is nonconstant and so has a
  nontrivial root.  For any knot in an \ZHS, a root of $\Delta_K$ gives
  rise to a reducible representation $\pi_K \to \SL$ with nonabelian
  image.  In the case of 2-bridge knots, the character of any such
  representation belongs to a component $X_0 \subset X(M)\nab$ (see
  Remark 1.9 and Corollary 2.9 of \cite{HildenLozanoMontesinos96},
  originating in Proposition 2.3 and the comment following it in
  \cite{Burde1990}), and Lemma~\ref{lem:consred} now finishes the
  proof.
\end{proof}

\begin{theorem}[{\cite[Lemma~4.8 and Theorem 4.9]{KM10}}]\label{thm:manytwobridge}
  Let $K = K(p, q)$ be a hyperbolic 2-bridge knot, and let
  $c$ be the lead coefficient of $\Delta_K$.  Suppose there exists a
  prime divisor $\ell$ of $p$ so that if $c \neq \pm 1$ then the
  reduction of $c$ mod $\ell$ is not in $\{-1, 0, 1\}$.  Then
  Conjecture~\ref{conj:char} holds for $K$.
\end{theorem}

\begin{proof}[Proof sketch]
  Let $X_0$ be any component of $X(K)\nab$. First, one shows that
  $X_0$ contains a character $\chi$ where $\tr(\mu_K) = 0$.  In
  Lemma~4.6 of \cite{KM10} this is shown using the particular
  structure of $\pi_K$, and it also follows from the following more
  general fact:
  \begin{claim} \label{claim:trval}
   Let $K$ be a knot in $S^3$ whose exterior contains no closed
    essential surface.  If $X_0$ is a component of $X(M)$, then given
    $a \in \C$ there is a $\chi \in X_0$ where $\tr(\mu_K) = a$.
 \end{claim}
 Two-bridge knots satisfy the hypothesis of Claim~\ref{claim:trval} by
 \cite{HatcherThurston1985}, and the proof of the claim
 is straightforward from the Culler-Shalen theory of surfaces
 associated to ideal points of $X_0$.  Specifically, on the smooth
 projective model of $X_0$ the rational function $\tr(\mu_K)$ takes on the
 value $a$ somewhere, and if this were at an ideal point the
 associated essential surface would have to be either closed or have
 meridian boundary; the latter situation also implies the existence of a closed
 essential surface by \cite[Theorem~2.0.3]{CGLS}.

 The representation corresponding to a $\chi$ where $\tr(\mu_K) = 0$
 is irreducible but has metabelian image, and in the 2-bridge case
 one can use this to calculate $\tk^\chi$ explicitly.  In particular,
 in \cite{KM10} they find that, provided there exists a prime
 $\ell$ as in the hypothesis, the polynomial $\tk^\chi$ is non-monic
 and $\deg \tk^\chi = 2 x(K)$.   We then apply
 Corollary~\ref{cor:chartorsion} to see that
 Conjecture~\ref{conj:char} holds.
\end{proof}

Another interesting class of characters in $X(M)\nab$ are those of
representations where $\mu_K$ is parabolic (e.g.~the distinguished
representation); such parabolic representations must occur on any
component $X_0$ by Claim~\ref{claim:trval}.  For the $3{,}830$
non-fibered 2-bridge knots with $q < p \leq 287$, we numerically
computed $\tk^\chi$ for all such parabolic characters, using a
precision of 150 decimal places.  In every case, the polynomial
$\tk^\chi$ was nonmonic and gave a sharp genus bound.  Since 2-bridge
knots contain no closed essential surfaces, every component of $X(M)$
is a curve.  Thus for all 2-bridge knots with $p \leq 287$ there are
only finitely many $\chi \in X(M)$ where $\tk^\chi$ is monic or where
$\deg(\tk^\chi) < 2 \deg(\Delta_K) - 2$, as conjectured in
\cite{KM10}.

\section{Character variety examples}\label{sec:charexs}

As with many things related to the character variety, while $\tkx$ is a
very natural concept, actually computing it can be difficult.  Here,
we content ourselves with finding $\tkx$ for three of the simplest
examples.  In each case, there is only one natural choice for $X_0$,
and moreover it is isomorphic to $\C \setminus \{ \mbox{finite set}
\}$.  Thus $X_0$ is rational and $\C(X_0)$ is just rational functions
in one variable, which makes it easy to express the answer.  We do one
fibered example and two that are nonfibered; in all cases the simplest
Seifert surface has genus 1.

\subsection{Example: m003}

We start with the sibling $M$ of the figure-8 complement, which is one
of the two orientable cusped hyperbolic \3-manifolds of minimal
volume.  The manifold is $m003$ in the SnapPea census
\cite{CallahanHildebrandWeeks1999,SnapPy}, and is the once-punctured
torus bundle over the circle with monodromy
$\mysmallmatrix{-2}{1}{1}{-1}$.  Its homology is $H_1(M ; \Z) = \Z
\oplus \Z/5\Z$, and it is, for instance, the complement of a
null-homologous knot in $L(5,1)$.  After randomizing the triangulation
a bit, SnapPy gives the following presentation
\[
\pi \assign \pi_1(M) = \spandef{a,b}{ bab^3aba^{-2} = 1}.
\]
We will view $X(\pi)$ as a subvariety of $X\big(\FreeGroup{a,b}\big)$,
where $\FreeGroup{a,b}$ is the free group on $\{a,b\}$.  Now
$X\big(\FreeGroup{a,b}\big) \cong \C^3$ where the coordinates are
$(x,y,z) = \big(\tr(a), \tr(b), \tr(ab)\big)$; this is because the
trace of any word $w \in \FreeGroup{a,b}$ can be expressed in terms of
these coordinates using the fundamental relation $\tr(UV) = \tr(U)
\tr(V) - \tr(U V^{-1})$ for $U, V \in \SL$.  Since $\pi$ is defined by
the single relator $R = bab^3aba^{-2}$, the character variety $X(\pi)$
is cut out by the polynomials corresponding to $\tr(R) = 2$, $\tr\big(
[a, R] \big) = 2$, and $\tr\big( [b, R] \big) = 2$.  Using Gr\"obner
bases in \cite{SAGE} to decompose $X(\pi)$ into irreducible components
over $\Q$, we found a unique component $X_0$ which contains an
irreducible character, i.e.~contains a point where $\tr\big( [a,b]
\big) \neq 2$.  Explicitly, the ideal of $X_0$ is $\pair{yz - x - z,
  xz + 1}$ and hence $X_0$ can be bijectively parameterized by
\[
f \maps \C \setminus \{ 0 \} \to X_0  \mtext{where} f(u) = \left( u, 1
  - u^2, -1/u \right).
\]
To compute $\tkx$, we consider the curve $R_0 \subset R(\pi)$ lying
above $X_0$ consisting of representations of the form
\[
\rho(a) = \left( \begin{array}{cc} u & 1 \\ -1 & 0 \end{array} \right)
\mtext{and} \rho(b) =  \left( \begin{array}{cc} 0 & v \\ -v^{-1}
    & 1- u^2 \end{array} \right) \mtext{where $v + v^{-1} = u^{-1}$.}
\]
Such representations are parameterized by $v \in \C \setminus \{0\}$,
and hence $\C(R_0) \cong \C(v)$ and we have an explicit $\pi \to
\GL\left(2, \C(v)\right)$ which is the restriction of the tautological
representation.  Using Lemma~\ref{prop:fox}, we find the torsion
polynomial of this representation is
\[
t - \frac{2 \left(v^{4} + v^{2} + 1\right)}{v^{3} + v}  + t^{-1}.
\]
Substituting in $v = \pm \big(1 - \sqrt{-4 u^2 + 1}\big)\big/2u$ to
eliminate $v$ gives the final answer
\[
\tkxt = t + \frac{2\left(u^{2} - 1\right)}{u} + t^{-1}.
\]

\subsection{Example: m006}

The census manifold $M = m006$ can also be described as $5/2$ surgery
on one component of the Whitehead link $L$.  (Here our conventions are
such that $+1$ surgery on either component of $L$ gives the trefoil
knot whereas $-1$ surgery gives the figure-8 knot).  Thus $M$ is, for
instance, the complement of a null-homologous knot in the lens space
$L(5,2)$ and again $H_1(M ; \Z) = \Z \oplus \Z/5\Z$. Using
spun-normal surfaces, it is easy to check via \cite{t3m} that there is
a Seifert surface in $M$ which has genus one with one boundary component.

SnapPy gives the presentation
\[
\pi \assign \pi_1(M) = \spandef{a,b}{ b^2abab^2a^{-2} = 1}.
\]
Changing the generators to $a_{\mathit{new}} = a^{-1}$ and
$b_{\mathit{new}} = ab$ rewrites this as
\[
\pi =  \spandef{a,b}{ a^3 b a b^3 a b = 1}.
\]
Using the same setup as in the last example, we find a single component
$X_0$ containing an irreducible character.  The ideal of $X_0$ is
$\pair{ x - y, y^2 z - z - 1}$ and hence we can bijectively
parameterize $X_0$ by
\begin{equation} \label{eq:m006param}
f \maps \C \setminus \{1, -1\}  \to X_0 \mtext{where} f(u) = \big(u, u,
(u^2-1)^{-1} \big).
\end{equation}
Considering the curve of representations given by
\[
\rho(a) = \left( \begin{array}{cc} v & 1 \\ 0 & v^{-1} \end{array} \right)
\mtext{and} \rho(b) =  \left( \begin{array}{cc} v^{-1} & 0\\ \frac{3 -
      2u^2}{u^2 -1}
    & v \end{array} \right) \mtext{where $v + v^{-1} = u$}
\]
and again directly applying Lemma~\ref{prop:fox} and eliminating $v$ gives
\[
\tkx(t) = \frac{2 u^{2} - 1}{u^{2} - 1} \left(t + t^{-1}\right) + \frac{2 u^{3}}{u^{2} - 1}.
\]

\subsection{m037} The census manifold $M = m037$ has $H_1(M ; \Z) = \Z
\oplus \Z/8\Z$, and so is not a knot in a $\Z/2$-homology sphere.
However, this makes no difference in this character variety context.
Again, using spun-normal surfaces one easily checks that there is a
Seifert surface in $M$ which has genus one with one boundary component.
Now $\pi = \spandef{a,b}{a^3ba^2ba^3b^{-2}=1}$, and this time, there
are two components of $X(\pi)$ containing irreducible characters.
However, one of these consists entirely of metabelian representations
which factor through the epimorphism $\pi \to \Z/2 * \Z/2 =
\spandef{c,d}{c^2=d^2=1}$ given by $a \mapsto c$ and $b \mapsto d$.
Focusing on the other component $X_0$, it turns out the ideal is
$\pair{xz - 2y, 4 y^2 - z^2 - 4}$ and so we can parameterize $X_0$ by
\[
f \maps \C \setminus \{ -2, 0, 2\} \to X_0 \mtext{where} f(u) = \left(
  \frac{u^2+4}{4u} , \frac{u^2 + 4}{u^2 - 4}, \frac{8u}{u^2 - 4}\right)
\]
and then calculate
\[
\tkx(t) = \frac{{\left(u + 2\right)}^4}{16 u^2} \left( t +
  t^{-1}\right) + \frac{{\left(u + 2\right)} {\left(u^{4} + 4  u^{3} - 8  u^{2} + 16  u + 16\right)}}{8 \, {\left(u - 2\right)} u^{2}}.
\]

\subsection{The role of ideal points.}  A key part of the
Culler-Shalen theory is the association of an essential surface in the
manifold $M$ to each ideal point of a curve $X_0 \subset X(M)$.  The
details are in e.g.~\cite{Sh02}, but in brief consider the smooth
projective model $\Xhat$ with its rational map $\Xhat \to X_0$.  Now
$\Xhat$ is a smooth Riemann surface, and the finite number of points
where $\Xhat \to X_0$ is undefined are called the \emph{ideal points}
of $X_0$.  To each such point $x$, there is an associated non-trivial
action of $\pi \assign \pi_1(M)$ on a simplicial tree $T_x$.  One then
constructs a surface $S$ in $M$ dual to this action, which can be
taken to be essential (i.e.~incompressible, boundary incompressible,
and not boundary parallel).  As minimal complexity Seifert surfaces
often arise from an ideal point of some $X_0$, a very natural idea is
thus to try to use such an ideal point $x$ to say something about
$\tkx$.  Moreover, provided that $X_0 \subset X(M)\nab$, a surface
associated to an ideal point is never a fiber or semifiber, which
suggests that one might hope to prove non-monotonicity of $\tkx$ by
examining $\tkx(x)$.  Thus, we now compute what happens to $\tkx$ at
such ideal points in our two non-fibered examples $m006$ and $m037$.
(Aside: It is known that even for knots in $S^3$ not all boundary
slopes need arise from ideal points \cite{ChesebroTillmann2007}, so it
is probably too much to expect that there is always an ideal point
which gives a Seifert surface.)

\subsection{Ideal points of $m006$}  If we view the parameterization
(\ref{eq:m006param}) above as a rational map from $\POne \to X_0$, we
have ideal points corresponding to $u \in \{-1, 1, \infty\}$.  To
calculate the boundary slopes of the surfaces associated to each of
these, we consider the trace functions of SnapPy's preferred basis
$\mu, \lambda$ for $\pi_1(\partial M)$.  In our presentation for
$\pi$, we calculate
\begin{align*}
\tr( \mu ) &= \tr( a^2 b a b ) = x \left(z^2 - z - 1\right) =
-\frac{u{\left(u^{4} - u^{2} - 1\right)}}{{\left(u - 1\right)}^{2} {\left(u + 1\right)}^{2}}  \\
\tr(\lambda ) &= \tr(a^3 b a) =-x^3 - x z + 2 x = -\frac{u {\left(u^{4} - 3 \, u^{2} + 3\right)}}{{\left(u - 1\right)} {\left(u + 1\right)}} \\
\tr(\mu\lambda) &= x^4 - 2 x^2 - z + 1 = \frac{{\left(u^{2} - 2\right)} {\left(u^{4} - u^{2} + 1\right)}}{{\left(u - 1\right)} {\left(u + 1\right)}}.
\end{align*}
Now, consider an ideal point $x$ with associated surface $S$, and pick
a simple closed curve $\gamma$ on $\pi_1(\partial M)$.  Then the
number of times $\gamma$ intersects $\partial S$ is twice the order of
the pole of $\tr(\gamma)$ at $x$ (here, if $\tr(\gamma)$ has a zero of
order $m$ at $x$, this counts as a pole of order $0$, not one of order
$-m$).  The above formulae thus show that the points $u=1$ and $u=-1$
give surfaces with boundary slope $\mu \lambda^2$, whereas $u=\infty$
gives one with boundary slope $\mu^3 \lambda^{-1}$.  The latter is the
homological longitude, and as there is only one spun-normal surface
with that boundary slope and each choice of spinning direction, it
follows that the surface associated to $\xi = \infty$ must be the
minimal genus Seifert surface.  Thus we're interested in
\[
\tkx(u=\infty)(t) = 2 \left( t + t^{-1} \right) + (\mbox{simple pole}) t^0.
\]

\subsection{Ideal points of $m037$}\label{sec:m037}

This time we have four ideal points corresponding to $u=-2,2,0,
\infty$.  We find
\begin{align*}
\tr(\mu) &= \tr(a^2 b a^3) = \frac{u^{8} - 48  u^{6} + 96  u^{4} - 768  u^{2} + 256}{64
    {\left(u - 2\right)} {\left(u + 2\right)} u^{3}} \\
\tr(\lambda) &= \tr(a^{-1} b a^3 b) \\
 &=  -\frac{u^{12} - 72  u^{10} + 1264  u^{8} - 12032  u^{6} + 20224
   u^{4} - 18432  u^{2} + 4096}{256  {\left(u - 2\right)}^{2} {\left(u
       + 2\right)}^{2} u^{4}} \\
\tr(\mu \lambda) &= -\frac{u^{8} - 16  u^{6} + 352  u^{4} - 256  u^{2} + 256}{4  {\left(u - 2\right)}^{3} {\left(u + 2\right)}^{3} u}.
\end{align*}
Hence $\{2, -2\}$ give surfaces with boundary slope $\mu^2
\lambda^{-1}$ and $\{0,\infty \}$ give surfaces with boundary slope
$\mu^4 \lambda^{3}$.  In fact, the homological longitude is $\mu^2
\lambda^{-1}$ and again using spun-normal surfaces one easily checks
that surfaces associated to $\{2, -2\}$ are the minimal genus Seifert
surface.  Thus, we care about
\[
\tkx(u=2)(t) = 4 \left(t + t^{-1}\right) + (\mbox{simple pole}) t^0
\mtext {and} \tkx(u=-2)(t) = 0.
\]
\subsection{General picture for ideal points}

Based on the preceding examples and a heuristic calculation for
tunnel-number one manifolds, we posit:
\begin{conjecture}\label{conj:ideal}
  Let $K$ be a knot in a rational homology 3-sphere, and $X_0$ a
  component of $X(K)\nab$.  Suppose $x$ is an ideal point of $X_0$
  which gives a Seifert surface (hence $K$ is nonfibered).  Then the
  lead coefficient of $\tkx$ has a finite value at $x$.
\end{conjecture}
Unfortunately, Conjecture~\ref{conj:ideal} does not seem particularly
promising as an attack on Conjecture~\ref{conj:char} for
distinguishing fibered versus nonfibered cases.  Moreover, in terms of
looking at such Seifert ideal points to show that $\tkx$ determines
the genus, the second ideal point $u=-2$ in Section~\ref{sec:m037}
where $\tkx$ vanishes is not a promising sign.

However, when trying to use an ideal point $x$ of $X_0$ which gives a
Seifert surface to understand $\tkx$, it may be wrong to focus on just
the value of $\tkx$ at $x$.  After all, there is no representation of
$\pi$ corresponding to $x$.  Rather, as in the construction of the
surface associated to $x$, perhaps one should view $x$ as giving a
valuation on $\C(R_0)$, where $R_0$ is a component of $R(M)$
surjecting onto $X_0$.  If we unwind the definition of the associated
surface and its properties, we are left with the following abstract
situation.  There is a field $\F$ with an additive valuation $v \maps
\F^\times \to \Z$ with a representation $\rho \maps \pi \to \SLF$ so
that for each $\gamma \in \pi$ we have $v\left(\tr(\gamma)\right) \geq
\abs{\phi(\gamma)}$ where $\phi \maps \pi \to \Z$ is the usual free
abelianization homomorphism.  This alone is not enough, because even
in the fibered case one always has such a setup by looking at an ideal
point of a component of $X(M)$ consisting of reducible
representations.  Thus it seems that the key to such an approach must
be to exploit the fact that since $X_0$ contains an irreducible
character there is a $\gamma \in \pi$ with $\phi(\gamma) = 0$ yet
$v\left(\tr(\gamma)\right)$ is arbitrarily large.   

\newpage  
{\RaggedRight
\bibliographystyle{nmd/math}
\bibliography{ahyp_knot} }

\def\cprime{$'$}
\begin{thebibliography}{CGHN}

\bibitem[BV]{BV10}
N.~Bergeron and A.~Venkatesh.
\newblock {The asymptotic growth of torsion homology for arithmetic groups}.
\newblock Preprint 2010, 49 pages.
\newblock \href{http://arxiv.org/abs/arXiv:1004.1083}{{\tt arXiv:1004.1083}}.

\bibitem[Bur1]{Burde1990}
G.~Burde.
\newblock \href{http://dx.doi.org/10.1007/BF01444524}{{{${\rm
  SU}(2)$}-representation spaces for two-bridge knot groups}}.
\newblock {\em Math. Ann.} {\bf 288} (1990), 103--119.
\newblock \mathreviewsnumber{1070927 (92c:57004)}.

\bibitem[Bur2]{Regina}
B.~Burton.
\newblock {Regina, a normal surface theory calculator, version 4.6}, 2009.
\newblock \url{http://regina.sourceforge.net/}

\bibitem[CHW]{CallahanHildebrandWeeks1999}
P.~J. Callahan, M.~V. Hildebrand, and J.~R. Weeks.
\newblock \href{http://dx.doi.org/10.1090/S0025-5718-99-01036-4}{{A census of
  cusped hyperbolic {$3$}-manifolds}}.
\newblock {\em Math. Comp.} {\bf 68} (1999), 321--332.
\newblock With microfiche supplement.
\newblock \mathreviewsnumber{1620219 (99c:57035)}.

\bibitem[Cha]{Ch03}
J.~C. Cha.
\newblock \href{http://dx.doi.org/10.1090/S0002-9947-03-03348-8}{{Fibred knots
  and twisted {A}lexander invariants}}.
\newblock {\em Trans. Amer. Math. Soc.} {\bf 355} (2003), 4187--4200.
\newblock \href{http://arxiv.org/abs/arXiv:math/0109136}{{\tt
  arXiv:math/0109136}}, \mathreviewsnumber{1990582 (2004e:57008)}.

\bibitem[Che1]{Che77}
J.~Cheeger.
\newblock {Analytic torsion and {R}eidemeister torsion}.
\newblock {\em Proc. Nat. Acad. Sci. U.S.A.} {\bf 74} (1977), 2651--2654.
\newblock \mathreviewsnumber{0451312 (56 \#\#9599)}.

\bibitem[Che2]{Che79}
J.~Cheeger.
\newblock \href{http://dx.doi.org/10.2307/1971113}{{Analytic torsion and the
  heat equation}}.
\newblock {\em Ann. of Math. (2)} {\bf 109} (1979), 259--322.
\newblock \mathreviewsnumber{528965 (80j:58065a)}.

\bibitem[CT]{ChesebroTillmann2007}
E.~Chesebro and S.~Tillmann.
\newblock
  \href{http://projecteuclid.org/getRecord?id=euclid.cag/1208527882}{{Not all
  boundary slopes are strongly detected by the character variety}}.
\newblock {\em Comm. Anal. Geom.} {\bf 15} (2007), 695--723.
\newblock \href{http://arxiv.org/abs/arXiv:math/0510418}{{\tt
  arXiv:math/0510418}}, \mathreviewsnumber{2395254 (2009b:57044)}.

\bibitem[CL]{CL96}
D.~Cooper and D.~D. Long.
\newblock \href{http://dx.doi.org/10.1142/S0218216596000357}{{Remarks on the
  {$A$}-polynomial of a knot}}.
\newblock {\em J. Knot Theory Ramifications} {\bf 5} (1996), 609--628.
\newblock \mathreviewsnumber{1414090 (97m:57004)}.

\bibitem[CGHN]{SnapPaper}
D.~Coulson, O.~A. Goodman, C.~D. Hodgson, and W.~D. Neumann.
\newblock
  \href{http://projecteuclid.org/getRecord?id=euclid.em/1046889596}{{Computing
  arithmetic invariants of 3-manifolds}}.
\newblock {\em Experiment. Math.} {\bf 9} (2000), 127--152.
\newblock \mathreviewsnumber{1758805 (2001c:57014)}.

\bibitem[Cro]{Crowell1959}
R.~Crowell.
\newblock {Genus of alternating link types}.
\newblock {\em Ann. of Math. (2)} {\bf 69} (1959), 258--275.
\newblock \mathreviewsnumber{0099665 (20 \#\#6103b)}.

\bibitem[CF]{CF63}
R.~H. Crowell and R.~H. Fox.
\newblock {\em Introduction to knot theory}.
\newblock Based upon lectures given at Haverford College under the Philips
  Lecture Program. Ginn and Co., Boston, Mass., 1963.
\newblock \mathreviewsnumber{0146828 (26 \#\#4348)}.

\bibitem[CD]{t3m}
M.~Culler and N.~M. Dunfield.
\newblock {t3m: a {P}ython 3-manifold library}, 2010.
\newblock \url{http://t3m.computop.org}

\bibitem[CDW]{SnapPy}
M.~Culler, N.~M. Dunfield, and J.~R. Weeks.
\newblock {{S}nap{P}y, a computer program for studying the geometry and
  topology of 3-manifolds}.
\newblock \url{http://snappy.computop.org/}

\bibitem[CGLS]{CGLS}
M.~Culler, C.~M. Gordon, J.~Luecke, and P.~B. Shalen.
\newblock \href{http://dx.doi.org/10.2307/1971311}{{Dehn surgery on knots}}.
\newblock {\em Ann. of Math. (2)} {\bf 125} (1987), 237--300.
\newblock \mathreviewsnumber{881270 (88a:57026)}.

\bibitem[CS]{CS83}
M.~Culler and P.~B. Shalen.
\newblock \href{http://dx.doi.org/10.2307/2006973}{{Varieties of group
  representations and splittings of {$3$}-manifolds}}.
\newblock {\em Ann. of Math. (2)} {\bf 117} (1983), 109--146.
\newblock \mathreviewsnumber{683804 (84k:57005)}.

\bibitem[DY]{DY09}
J.~Dubois and Y.~Yamaguchi.
\newblock {Multivariable Twisted Alexander Polynomial for hyperbolic
  three-manifolds with boundary}.
\newblock Preprint 2009, 39 pages.
\newblock \href{http://arxiv.org/abs/arXiv:0906.1500}{{\tt arXiv:0906.1500}}.

\bibitem[DFJ]{software}
N.~M. Dunfield, S.~Friedl, and N.~Jackson.
\newblock {Software and data accompanying this paper}.
\newblock \url{http://dunfield.info/torsion}

\bibitem[DGST]{DGST2010}
N.~M. Dunfield, S.~Garoufalidis, A.~Shumakovitch, and M.~Thistlethwaite.
\newblock \href{http://nyjm.albany.edu:8000/j/2010/16_99.html}{{Behavior of
  knot invariants under genus 2 mutation}}.
\newblock {\em New York J. Math.} {\bf 16} (2010), 99--123.
\newblock \href{http://arxiv.org/abs/arXiv:math/0607258}{{\tt
  arXiv:math/0607258}}, \mathreviewsnumber{2657370}.

\bibitem[DR]{DR10}
N.~M. Dunfield and D.~Ramakrishnan.
\newblock \href{http://dx.doi.org/10.1353/ajm.0.0098}{{Increasing the number of
  fibered faces of arithmetic hyperbolic 3-manifolds}}.
\newblock {\em Amer. J. Math.} {\bf 132} (2010), 53--97.
\newblock \href{http://arxiv.org/abs/arXiv:0712.3243}{{\tt arXiv:0712.3243}},
  \mathreviewsnumber{2597506 (2011d:57047)}.

\bibitem[Fox1]{Fo53}
R.~H. Fox.
\newblock {Free differential calculus. {I}. {D}erivation in the free group
  ring}.
\newblock {\em Ann. of Math. (2)} {\bf 57} (1953), 547--560.
\newblock \mathreviewsnumber{0053938 (14,843d)}.

\bibitem[Fox2]{Fo54}
R.~H. Fox.
\newblock {Free differential calculus. {II}. {T}he isomorphism problem of
  groups}.
\newblock {\em Ann. of Math. (2)} {\bf 59} (1954), 196--210.
\newblock \mathreviewsnumber{0062125 (15,931e)}.

\bibitem[Fri]{Fr88}
D.~Fried.
\newblock \href{http://dx.doi.org/10.1007/BFb0081399}{{Cyclic resultants of
  reciprocal polynomials}}.
\newblock In {\em Holomorphic dynamics ({M}exico, 1986)}, volume 1345 of {\em
  Lecture Notes in Math.}, pages 124--128. Springer, Berlin, 1988.
\newblock \mathreviewsnumber{980956 (90h:57004)}.

\bibitem[FJ]{FJ11}
S.~Friedl and N.~Jackson.
\newblock {Approximations to the volume of hyperbolic knots}.
\newblock In T.~Morifuji, editor, {\em Twisted topological invariants and
  topology of low-dimensional manifolds}, volume 1747 of {\em RIMS
  K\^oky\^uroku}, pages 35--46, 2011.
\newblock \href{http://arxiv.org/abs/arXiv:1102.3742}{{\tt arXiv:1102.3742}}.

\bibitem[FK1]{FK06}
S.~Friedl and T.~Kim.
\newblock \href{http://dx.doi.org/10.1016/j.top.2006.06.003}{{The {T}hurston
  norm, fibered manifolds and twisted {A}lexander polynomials}}.
\newblock {\em Topology} {\bf 45} (2006), 929--953.
\newblock \href{http://arxiv.org/abs/arXiv:math/0505594}{{\tt
  arXiv:math/0505594}}, \mathreviewsnumber{2263219 (2007g:57020)}.

\bibitem[FK2]{FK08}
S.~Friedl and T.~Kim.
\newblock \href{http://dx.doi.org/10.1090/S0002-9947-08-04455-3}{{Twisted
  {A}lexander norms give lower bounds on the {T}hurston norm}}.
\newblock {\em Trans. Amer. Math. Soc.} {\bf 360} (2008), 4597--4618.
\newblock \href{http://arxiv.org/abs/arXiv:math/0505682}{{\tt
  arXiv:math/0505682}}, \mathreviewsnumber{2403698 (2010a:57021)}.

\bibitem[FKK]{FKK11}
S.~Friedl, T.~Kim, and T.~Kitayama.
\newblock {Poincar\'e duality and degrees of twisted Alexander polynomials}.
\newblock Preprint 2011, 42 pages.
\newblock \href{http://arxiv.org/abs/arXiv:1107.3002}{{\tt arXiv:1107.3002}}.

\bibitem[FV1]{FV11}
S.~Friedl and S.~Vidussi.
\newblock \href{http://dx.doi.org/10.4007/annals.2011.173.3.8}{{Twisted
  Alexander polynomials detect fibered 3-manifolds}}.
\newblock {\em Ann. of Math. (2)} {\bf 173} (2001), 1587--1643.
\newblock \href{http://arxiv.org/abs/arXiv:0805.1234}{{\tt arXiv:0805.1234}}.

\bibitem[FV2]{FV08b}
S.~Friedl and S.~Vidussi.
\newblock \href{http://dx.doi.org/10.1353/ajm.2008.0014}{{Twisted {A}lexander
  polynomials and symplectic structures}}.
\newblock {\em Amer. J. Math.} {\bf 130} (2008), 455--484.
\newblock \href{http://arxiv.org/abs/arXiv:math/0604398}{{\tt
  arXiv:math/0604398}}, \mathreviewsnumber{2405164 (2009b:57053)}.

\bibitem[FV3]{FV10}
S.~Friedl and S.~Vidussi.
\newblock {A survey of twisted Alexander polynomials}.
\newblock Preprint 2010, 42 pages.
\newblock \href{http://arxiv.org/abs/arXiv:0905.0591}{{\tt arXiv:0905.0591}}.

\bibitem[Gab1]{Gabai1986}
D.~Gabai.
\newblock \href{http://dx.doi.org/10.1007/BF02621931}{{Detecting fibred links
  in {$S^3$}}}.
\newblock {\em Comment. Math. Helv.} {\bf 61} (1986), 519--555.
\newblock \mathreviewsnumber{870705 (88c:57009)}.

\bibitem[Gab2]{Gabai1987}
D.~Gabai.
\newblock
  \href{http://projecteuclid.org/getRecord?id=euclid.jdg/1214441488}{{Foliations
  and the topology of {$3$}-manifolds. {III}}}.
\newblock {\em J. Differential Geom.} {\bf 26} (1987), 479--536.
\newblock \mathreviewsnumber{910018 (89a:57014b)}.

\bibitem[GKM]{GKM05}
H.~Goda, T.~Kitano, and T.~Morifuji.
\newblock \href{http://dx.doi.org/10.4171/CMH/3}{{Reidemeister torsion, twisted
  {A}lexander polynomial and fibered knots}}.
\newblock {\em Comment. Math. Helv.} {\bf 80} (2005), 51--61.
\newblock \href{http://arxiv.org/abs/arXiv:math/0311155}{{\tt
  arXiv:math/0311155}}, \mathreviewsnumber{2130565 (2005m:57008)}.

\bibitem[GN]{Snap}
O.~Goodman and W.~Neumann.
\newblock {Snap, a computer program for studying arithmetic invariants of
  hyperbolic 3-manifolds}.
\newblock \url{http://snap-pari.sourceforge.net/}

\bibitem[Har]{Ha05}
S.~L. Harvey.
\newblock \href{http://dx.doi.org/10.1016/j.top.2005.03.001}{{Higher-order
  polynomial invariants of 3-manifolds giving lower bounds for the {T}hurston
  norm}}.
\newblock {\em Topology} {\bf 44} (2005), 895--945.
\newblock \href{http://arxiv.org/abs/arXiv:math/0207014}{{\tt
  arXiv:math/0207014}}, \mathreviewsnumber{2153977 (2006g:57019)}.

\bibitem[HT]{HatcherThurston1985}
A.~Hatcher and W.~Thurston.
\newblock \href{http://dx.doi.org/10.1007/BF01388971}{{Incompressible surfaces
  in {$2$}-bridge knot complements}}.
\newblock {\em Invent. Math.} {\bf 79} (1985), 225--246.
\newblock \mathreviewsnumber{778125 (86g:57003)}.

\bibitem[HLM]{HildenLozanoMontesinos96}
H.~M. Hilden, M.~T. Lozano, and J.~M. Montesinos-Amilibia.
\newblock {On the arithmetic $2$-bridge knots and link orbifolds and a new knot
  invariant}.
\newblock {\em J. Knot Theory Ramifications} {\bf 4} (1995), 81--114.
\newblock \mathreviewsnumber{96a:57019}.

\bibitem[Hil]{Hillar2005a}
C.~J. Hillar.
\newblock \href{http://dx.doi.org/10.1016/j.jsc.2005.01.001}{{Cyclic
  resultants}}.
\newblock {\em J. Symbolic Comput.} {\bf 39} (2005), 653--669.
\newblock \href{http://arxiv.org/abs/arXiv:math/0401220}{{\tt
  arXiv:math/0401220}}, \mathreviewsnumber{2167674 (2006h:12001)}.

\bibitem[HL]{HillarLevine2007}
C.~J. Hillar and L.~Levine.
\newblock \href{http://dx.doi.org/10.1090/S0002-9939-06-08672-2}{{Polynomial
  recurrences and cyclic resultants}}.
\newblock {\em Proc. Amer. Math. Soc.} {\bf 135} (2007), 1607--1618.
\newblock \href{http://arxiv.org/abs/arXiv:math/0411414}{{\tt
  arXiv:math/0411414}}, \mathreviewsnumber{2286068 (2007m:11015)}.

\bibitem[HSW]{HSW10}
J.~A. Hillman, D.~S. Silver, and S.~G. Williams.
\newblock \href{http://dx.doi.org/10.2140/agt.2010.10.1017}{{On reciprocality
  of twisted {A}lexander invariants}}.
\newblock {\em Algebr. Geom. Topol.} {\bf 10} (2010), 1017--1026.
\newblock \href{http://arxiv.org/abs/arXiv:0905.2574}{{\tt arXiv:0905.2574}},
  \mathreviewsnumber{2629774}.

\bibitem[HM]{HM08}
M.~Hirasawa and K.~Murasugi.
\newblock \href{http://dx.doi.org/10.1142/S0218216510008418}{{Evaluations of
  the twisted {A}lexander polynomials of 2-bridge knots at {$\pm 1$}}}.
\newblock {\em J. Knot Theory Ramifications} {\bf 19} (2010), 1355--1400.
\newblock \href{http://arxiv.org/abs/arXiv:0808.3058}{{\tt arXiv:0808.3058}},
  \mathreviewsnumber{2735813}.

\bibitem[HT]{Knotscape}
J.~Hoste and M.~Thistlethwaite.
\newblock {Knotscape}, 1999.
\newblock \url{http://pzacad.pitzer.edu/~jhoste/HosteWebPages/kntscp.html}

\bibitem[HTW]{HTW98}
J.~Hoste, M.~Thistlethwaite, and J.~Weeks.
\newblock \href{http://dx.doi.org/10.1007/BF03025227}{{The first 1,701,936
  knots}}.
\newblock {\em Math. Intelligencer} {\bf 20} (1998), 33--48.
\newblock \mathreviewsnumber{1646740 (99i:57015)}.

\bibitem[KM1]{KM10}
T.~Kim and T.~Morifuji.
\newblock {Twisted Alexander polynomials and character varieties of 2-bridge
  knot groups}.
\newblock Preprint 2010, 19 pages.
\newblock \href{http://arxiv.org/abs/arXiv:1006.4285}{{\tt arXiv:1006.4285}}.

\bibitem[Kit1]{Ki96}
T.~Kitano.
\newblock
  \href{http://projecteuclid.org/getRecord?id=euclid.pjm/1102365178}{{Twisted
  {A}lexander polynomial and {R}eidemeister torsion}}.
\newblock {\em Pacific J. Math.} {\bf 174} (1996), 431--442.
\newblock \mathreviewsnumber{1405595 (97g:57007)}.

\bibitem[KM2]{KtM05}
T.~Kitano and T.~Morifuji.
\newblock {Divisibility of twisted {A}lexander polynomials and fibered knots}.
\newblock {\em Ann. Sc. Norm. Super. Pisa Cl. Sci. (5)} {\bf 4} (2005),
  179--186.
\newblock \mathreviewsnumber{2165406 (2006e:57006)}.

\bibitem[Kit2]{Kiy08}
T.~Kitayama.
\newblock {Normalization of twisted Alexander invariants}.
\newblock Preprint 2009, 19 pages.
\newblock \href{http://arxiv.org/abs/arXiv:0705.2371}{{\tt arXiv:0705.2371}}.

\bibitem[Lin]{Li01}
X.~S. Lin.
\newblock \href{http://dx.doi.org/10.1007/PL00011612}{{Representations of knot
  groups and twisted {A}lexander polynomials}}.
\newblock {\em Acta Math. Sin. (Engl. Ser.)} {\bf 17} (2001), 361--380.
\newblock \mathreviewsnumber{1852950 (2003f:57018)}.

\bibitem[MR]{MR03}
C.~Maclachlan and A.~W. Reid.
\newblock {\em The arithmetic of hyperbolic 3-manifolds}, volume 219 of {\em
  Graduate Texts in Mathematics}.
\newblock Springer-Verlag, New York, 2003.
\newblock \mathreviewsnumber{1937957 (2004i:57021)}.

\bibitem[MFP1]{MP09}
P.~Menal-Ferrer and J.~Porti.
\newblock {Twisted cohomology for hyperbolic three manifolds}, Preprint 2010.
\newblock 27 pages.
\newblock \href{http://arxiv.org/abs/arXiv:1001.2242}{{\tt arXiv:1001.2242}}.

\bibitem[MFP2]{MenalFerrerPorti2011}
P.~Menal-Ferrer and J.~Porti.
\newblock {Mutation and SL(2,$\mathbb{C}$)-Reidemeister torsion for hyperbolic
  knots}, Preprint 2011.
\newblock 19 pages.
\newblock \href{http://arxiv.org/abs/arXiv:1108.2460}{{\tt arXiv:1108.2460}}.

\bibitem[Mil]{Mi66}
J.~Milnor.
\newblock {Whitehead torsion}.
\newblock {\em Bull. Amer. Math. Soc.} {\bf 72} (1966), 358--426.
\newblock \mathreviewsnumber{0196736 (33 \#\#4922)}.

\bibitem[Mor]{Mo08}
T.~Morifuji.
\newblock \href{http://dx.doi.org/10.1016/j.bulsci.2008.04.001}{{Twisted
  {A}lexander polynomials of twist knots for nonabelian representations}}.
\newblock {\em Bull. Sci. Math.} {\bf 132} (2008), 439--453.
\newblock \mathreviewsnumber{2426646 (2009f:57019)}.

\bibitem[Mu1]{Mu78}
W.~M{\"u}ller.
\newblock \href{http://dx.doi.org/10.1016/0001-8708(78)90116-0}{{Analytic
  torsion and {$R$}-torsion of {R}iemannian manifolds}}.
\newblock {\em Adv. in Math.} {\bf 28} (1978), 233--305.
\newblock \mathreviewsnumber{498252 (80j:58065b)}.

\bibitem[Mu2]{Mu09}
W.~M{\"u}ller.
\newblock \href{http://www.mpim-bonn.mpg.de/webfm_send/5}{{Analytic torsion and
  cohomology of hyperbolic 3-manifolds}}.
\newblock Preprint of the Max-Planck Institute, Bonn (2009).

\bibitem[Mu3]{Mu10}
W.~M\"uller.
\newblock {The asymptotics of the Ray-Singer analytic torsion of hyperbolic
  3-manifolds}.
\newblock Preprint 2010, 35 pages.
\newblock \href{http://arxiv.org/abs/arXiv:1003.5168}{{\tt arXiv:1003.5168}}.

\bibitem[Mur1]{Murasugi1958}
K.~Murasugi.
\newblock {On the genus of the alternating knot. {I}, {II}}.
\newblock {\em J. Math. Soc. Japan} {\bf 10} (1958), 94--105, 235--248.
\newblock \mathreviewsnumber{0099664 (20 \#\#6103a)}.

\bibitem[Mur2]{Murasugi1963}
K.~Murasugi.
\newblock {On a certain subgroup of the group of an alternating link}.
\newblock {\em Amer. J. Math.} {\bf 85} (1963), 544--550.
\newblock \mathreviewsnumber{0157375 (28 \#\#609)}.

\bibitem[NR]{NeumannReid1992}
W.~D. Neumann and A.~W. Reid.
\newblock
  \href{http://www.math.columbia.edu/~neumann/preprints/nrarith.pdf}{{Arithmetic
  of hyperbolic manifolds}}.
\newblock In {\em Topology '90 ({C}olumbus, {OH}, 1990)}, volume~1 of {\em Ohio
  State Univ. Math. Res. Inst. Publ.}, pages 273--310. de Gruyter, Berlin,
  1992.
\newblock \mathreviewsnumber{1184416 (94c:57024)}.

\bibitem[Nic]{Nic03}
L.~I. Nicolaescu.
\newblock \href{http://dx.doi.org/10.1515/9783110198102}{{\em The
  {R}eidemeister torsion of 3-manifolds}}, volume~30 of {\em de Gruyter Studies
  in Mathematics}.
\newblock Walter de Gruyter \& Co., Berlin, 2003.
\newblock \mathreviewsnumber{1968575 (2004e:57018)}.

\bibitem[Paj]{Paj07}
A.~V. Pajitnov.
\newblock \href{http://dx.doi.org/10.1090/S1061-0022-07-00975-2}{{Novikov
  homology, twisted {A}lexander polynomials, and {T}hurston cones}}.
\newblock {\em Algebra i Analiz} {\bf 18} (2006), 173--209.
\newblock \mathreviewsnumber{2301045 (2007k:57023)}.

\bibitem[Par]{Par09}
J.~Park.
\newblock \href{http://dx.doi.org/10.1016/j.jfa.2009.06.012}{{Analytic torsion
  and {R}uelle zeta functions for hyperbolic manifolds with cusps}}.
\newblock {\em J. Funct. Anal.} {\bf 257} (2009), 1713--1758.
\newblock \mathreviewsnumber{2540990}.

\bibitem[Por1]{Po09}
J.~Porti.
\newblock Private communication, 2009.

\bibitem[Por2]{Po97}
J.~Porti.
\newblock {Torsion de {R}eidemeister pour les vari\'et\'es hyperboliques}.
\newblock {\em Mem. Amer. Math. Soc.} {\bf 128} (1997), x+139.
\newblock \mathreviewsnumber{1396960 (98g:57034)}.

\bibitem[Rag]{Ra65}
M.~S. Raghunathan.
\newblock {On the first cohomology of discrete subgroups of semisimple {L}ie
  groups}.
\newblock {\em Amer. J. Math.} {\bf 87} (1965), 103--139.
\newblock \mathreviewsnumber{0173730 (30 \#\#3940)}.

\bibitem[RS]{RS71}
D.~B. Ray and I.~M. Singer.
\newblock {{$R$}-torsion and the {L}aplacian on {R}iemannian manifolds}.
\newblock {\em Advances in Math.} {\bf 7} (1971), 145--210.
\newblock \mathreviewsnumber{0295381 (45 \#\#4447)}.

\bibitem[Rub]{Ru87}
D.~Ruberman.
\newblock \href{http://dx.doi.org/10.1007/BF01389038}{{Mutation and volumes of
  knots in {$S\sp 3$}}}.
\newblock {\em Invent. Math.} {\bf 90} (1987), 189--215.
\newblock \mathreviewsnumber{906585 (89d:57018)}.

\bibitem[Sha]{Sh02}
P.~B. Shalen.
\newblock \href{http://www.math.uic.edu/~shalen/handbook.ps}{{Representations
  of 3-manifold groups}}.
\newblock In {\em Handbook of geometric topology}, pages 955--1044.
  North-Holland, Amsterdam, 2002.
\newblock \mathreviewsnumber{1886685 (2003d:57002)}.

\bibitem[SW]{SW09c}
D.~S. Silver and S.~G. Williams.
\newblock \href{http://dx.doi.org/10.1016/j.topol.2009.08.016}{{Dynamics of
  twisted {A}lexander invariants}}.
\newblock {\em Topology Appl.} {\bf 156} (2009), 2795--2811.
\newblock \href{http://arxiv.org/abs/arXiv:0801.2118}{{\tt arXiv:0801.2118}},
  \mathreviewsnumber{2556037 (2011a:57028)}.

\bibitem[Sage]{SAGE}
W.~Stein et~al.
\newblock {\em {Sage} {M}athematics {S}oftware, Version 4.7}.
\newblock \url{http://www.sagemath.org}

\bibitem[Sto]{St10}
A.~Stoimenow.
\newblock {Knot data tables}.
\newblock \url{http://stoimenov.net/stoimeno/homepage/ptab/}

\bibitem[ST]{StTa09}
A.~Stoimenow and T.~Tanaka.
\newblock {Mutation and the colored Jones polynomial}.
\newblock {\em J. G\"okova Geom. Topol. GGT} {\bf 3} (2009), 44--78.
\newblock \href{http://arxiv.org/abs/arXiv:math/0607794}{{\tt
  arXiv:math/0607794}}, \mathreviewsnumber{2595755 (2011a:57029)}.

\bibitem[Thu]{Th97}
W.~P. Thurston.
\newblock {\em Three-dimensional geometry and topology. {V}ol. 1}, volume~35 of
  {\em Princeton Mathematical Series}.
\newblock Princeton University Press, Princeton, NJ, 1997.
\newblock Edited by Silvio Levy.
\newblock \mathreviewsnumber{1435975 (97m:57016)}.

\bibitem[Til1]{Ti00}
S.~Tillmann.
\newblock \href{http://dx.doi.org/10.1142/S0218216500000311}{{On the
  {K}inoshita-{T}erasaka knot and generalised {C}onway mutation}}.
\newblock {\em J. Knot Theory Ramifications} {\bf 9} (2000), 557--575.
\newblock \mathreviewsnumber{1758873 (2002a:57013)}.

\bibitem[Til2]{Ti04}
S.~Tillmann.
\newblock \href{http://dx.doi.org/10.2140/agt.2004.4.133}{{Character varieties
  of mutative 3-manifolds}}.
\newblock {\em Algebr. Geom. Topol.} {\bf 4} (2004), 133--149.
\newblock \href{http://arxiv.org/abs/math.GT/0306053}{{\tt math.GT/0306053}},
  \mathreviewsnumber{2059186 (2005c:57016)}.

\bibitem[Tur1]{Tu86}
V.~G. Turaev.
\newblock {Reidemeister torsion in knot theory}.
\newblock {\em Uspekhi Mat. Nauk} {\bf 41} (1986), 97--147, 240.
\newblock \mathreviewsnumber{832411 (87i:57009)}.

\bibitem[Tur2]{Tu90}
V.~G. Turaev.
\newblock {Euler structures, nonsingular vector fields, and {R}eidemeister-type
  torsions}.
\newblock {\em Izv. Akad. Nauk SSSR Ser. Mat.} {\bf 53} (1989), 607--643, 672.
\newblock \mathreviewsnumber{1013714 (90m:57021)}.

\bibitem[Tur3]{Tu01}
V.~Turaev.
\newblock {\em Introduction to combinatorial torsions}.
\newblock Lectures in Mathematics ETH Z\"urich. Birkh\"auser Verlag, Basel,
  2001.
\newblock Notes taken by Felix Schlenk.
\newblock \mathreviewsnumber{1809561 (2001m:57042)}.

\bibitem[Tur4]{Tu02}
V.~Turaev.
\newblock {\em Torsions of 3-dimensional manifolds}, volume 208 of {\em
  Progress in Mathematics}.
\newblock Birkh\"auser Verlag, Basel, 2002.
\newblock \mathreviewsnumber{1958479 (2003m:57028)}.

\bibitem[Wada]{Wad94}
M.~Wada.
\newblock \href{http://dx.doi.org/10.1016/0040-9383(94)90013-2}{{Twisted
  {A}lexander polynomial for finitely presentable groups}}.
\newblock {\em Topology} {\bf 33} (1994), 241--256.
\newblock \mathreviewsnumber{1273784 (95g:57021)}.

\bibitem[Wal]{Wal78}
F.~Waldhausen.
\newblock {Algebraic {$K$}-theory of generalized free products. {I}, {II}}.
\newblock {\em Ann. of Math. (2)} {\bf 108} (1978), 135--204.
\newblock \mathreviewsnumber{0498807 (58 \#\#16845a)}.

\end{thebibliography}
\end{document}